\documentclass[12pt,a4paper]{article}
\usepackage[latin1]{inputenc}
\usepackage[english]{babel}
\usepackage{amsmath}
\usepackage{amsthm}
\usepackage{amsfonts}
\usepackage{amssymb}
\usepackage{authblk}
\usepackage[square,numbers]{natbib}
\usepackage{textcomp}
\usepackage[hmargin=2cm,paper=a4paper]{geometry}
\author{Henri Elad Altman \footnote{Email: henri.elad$\_$altman$@$upmc.fr}}
\affil{\footnotesize{Universit\'{e} Pierre et Marie Curie, LPMA, 4 Pl. Jussieu, 75005 Paris}}
\title{Bismut-Elworthy-Li formulae for Bessel processes}
\date{}
\theoremstyle{definition}

\newtheorem{nota}{Notation}
\newtheorem{rk}{Remark}
\theoremstyle{plain}
\newtheorem{lm}{Lemma}
\newtheorem{prop}{Proposition}
\newtheorem{thm}{Theorem}
\newtheorem{cor}{Corollary}

\begin{document}
\maketitle

\abstract
In this article we are interested in the differentiability property of the Markovian semi-group corresponding to the Bessel processes of nonnegative dimension. More precisely, for all $\delta \geq 0$ and $T>0$, we compute the derivative of the function $x \mapsto P^{\delta}_{T} F (x) $, where $(P^{\delta}_{t})_{t \geq 0}$ is the transition semi-group associated to the $\delta$ - dimensional Bessel process, and $F$ is any bounded Borel function on $\mathbb{R}_{+}$. The obtained expression shows a nice interplay between the transition semi-groups of the $\delta$ - and the $(\delta + 2)$-dimensional Bessel processes. As a consequence, we deduce that the Bessel processes satisfy the strong Feller property, with a continuity modulus which is independent of the dimension. Moreover, we provide a probabilistic interpretation of this expression as a Bismut-Elworthy-Li formula.

%



\section{Introduction}

Bessel processes are a one-parameter family of nonnegative diffusion processes with a singular drift, which present a reflecting behavior when they hit the origin. The smaller the parameter (called dimension), the more intense the reflection. Hence, studying the dynamics of these processes is a non-trivial problem, especially when the dimension is small. Despite these apparent difficulties, Bessel processes have remarkably nice properties. Therefore they provide an instructive insight in the study of stochastic differential equations (SDEs) with a singular drift, as well as the study of reflected SDEs. 

For all $x \geq 0$ and $\delta \geq 0$, the squared Bessel process of dimension $\delta$ started at $x^{2}$ is  the unique strong solution of the equation:
\begin{equation}
\label{sqbsde}
X_{t} = x^{2} + 2\int_{0}^{t} \sqrt{X_{s}} dB_{s} + \delta t.
\end{equation} 
Such a process $X$ is nonnegative, and the law of its square-root $\rho = \sqrt{X}$ is, by definition, the $\delta$-dimensional Bessel process started at $x$ (see \cite{revuz2013continuous}, section XI, or Chapter 3 of \cite{zambotti2017random} for an introduction to Bessel processes). The process $\rho$ satisfies the following SDE before its first hitting time $T_{0}$ of $0$:

\begin{equation*}
 \forall t \in [0, T_{0}), \qquad \rho_{t} = x + \frac{\delta - 1}{2} \int_{0}^{t} \frac{ds}{\rho_{s}} + B_{t}.
\end{equation*}

This is an SDE with non-Lipschitz continuous drift term given by the function $x \mapsto \frac{\delta-1}{2} \frac{1}{x}$ on $(0,+ \infty)$. Note that  when $ \delta < 1$, this function is nondecreasing on $\mathbb{R}_{+}$, and as $\delta$ decreases this "wrong" monotonicity becomes more and more acute. As a consequence, for $\delta$ small, the process $\rho$ is not mean-square differentiable, so that classical criteria for the Bismut-Elworthy-Li formula to hold (see \cite{cerrai2001second}, Section 1.5, and Section 2 below) do not apply here. Hence, one would not even expect such a formula to hold for $ \delta < 1$. For instance, even continuity of the flow is not known in this regime, see Remark \ref{rk5} below.

The aim of the present paper is to study the derivative in space of the family of transition kernels $(P^{\delta}_{T})_{T \geq 0}$ of the $\delta$-dimensional Bessel process. In a first part, we show that this derivative can be expressed in terms of the transition kernels of the $\delta$ - and the $(\delta +2)$-dimensional Bessel processes. More precisely, we prove that, for all function $F : \mathbb{R}_{+} \rightarrow  \mathbb{R}$ bounded and Borel, all $T>0$ and all $x \geq 0$, we have:

\begin{equation}
\label{eq1}
\frac{d}{dx}  P^{\delta}_{T} F (x) = \frac{x}{T} \left( P^{\delta +2}_{T} F (x) - P^{\delta}_{T} F (x) \right).
\end{equation}

As a consequence, the Bessel processes satisfy the strong Feller property uniformly in $\delta$.
In a second part, we interpret the above result probabilistically as a Bismut-Elworthy-Li formula. More precisely, given a realization $\rho$ of the Bessel process through the SDE \eqref{sqbsde}, we introduce the derivative  $\eta_{t}$ of $\rho_{t}$ with respect to the initial condition $x$, and show that when $\delta > 0$, the stochastic integral $\int_{0}^{t} \eta_{s} dB_{s}$ is well-defined as an $L^{p}$ martingale, for some $p > 1$ depending on $\delta$. Moreover, it turns out that $\int_{0}^{T} \eta_{s} dB_{s}$ is (up to a constant) the Radon-Nikodym derivative of the $(\delta+2)$-dimensional Bessel process over the interval $[0,T]$ w.r.t. the $\delta$-dimensional one. As a consequence, we deduce that the above equation can be rewritten:
\begin{equation} 
\label{eq2}
\frac{d}{dx}  P^{\delta}_{T} F (x) = \frac{1}{T} \mathbb{E} \left[ F(\rho_{t}(x)) \left( \int_{0}^{t} \eta_{s}(x) dB_{s} \right) \right] 
\end{equation}
which is an apparition, in an unexpected context, of the well-known Bismut-Elworthy-Li formula (see \cite{ELWORTHY1994252} for a precise statement and proof of the Bismut-Elworthy-Li formula in the case of diffusions with smooth coefficients).  

One surprising feature is that, while \eqref{eq1} is very easy to prove whatever the value of $\delta \geq 0$, on the other hand, the process $ (\int_{0}^{t} \eta_{s}(x) dB_{s})_{t\geq 0}$ has less and less finite moments as $\delta$ decreases, which makes the proof of \eqref{eq2} more involved for small $\delta$. In particular, this process is not in $L^{2}$ for $\delta < 2(\sqrt{2} -1)$, and when $\delta = 0$, we do not even know whether the stochastic integral $ \int_{0}^{t} \eta_{s}(x) dB_{s}$ is well-defined as a local martingale. 

This article was originally motivated by the hope to prove the strong Feller property for some singular reflected SDEs or SPDEs. Recently, several works have brought about new techniques to prove the strong Feller property for singular SPDEs. Thus, in \cite{tsatsoulis2016spectral}, the authors established this property for the $P(\Phi)_{2}$ equation, and in \cite{hairer2016strong}, the authors established it for a large class of singular semilinear SPDEs. The fact, mentioned above, that blowup of $\eta$ does not affect the strong Feller property of Bessel processes is reminiscent of the latter article, where the setting used to prove the strong Feller property allows blowup in finite time of the solution. Also, we hope that the techniques used in the present article might give inspiration to treat more general cases. Note that, even in the present context, where many computations can be performed explicitly, we still have an open problem concerning the Strong Feller bounds for Bessel processes of dimension $\delta \leq 2 (\sqrt{2} -1)$ (see Remark \ref{open_pbm} below).      

\medskip
The plan of our paper is as follows. In Section 2 we recall the classical Bismut-Elworthy-Li formula for diffusions in $\mathbb{R}$ with a dissipative drift, and show how this implies the strong Feller property. In Section 3 we recall the definition of Bessel processes and their basic properties. In Section 4 we compute the derivative of the Bessel semi-group. In section 5 we establish the differentiability of the Bessel flow at any given point in $\mathbb{R}^{*}_{+}$, and we give an expression for (some modification of) the derivative. In Section 6, we show that this derivative is not bounded in time when $\delta < 1$. We prove, however, that it is linked to an interesting martingale corresponding to the family of Radon-Nikodym derivatives of the $(\delta+2)$-dimensional Bessel process w.r.t.the $\delta$-dimensional one. In Section 7 we prove the Bismut-Elworthy-Li formula for the Bessel processes of dimension $\delta > 0$ . 


\section{Classical Bismut-Elworthy-Li formula for one-dimensional diffusions}

In this section we recall very briefly the Bismut-Elworthy-Li fomula in the case of one-dimensional diffusions, and the way this formula implies the strong Feller property. 

Consider an SDE on $\mathbb{R}$ of the form :
\begin{equation}
\label{one_dim_diff}
dX_{t} = b(X_{t}) dt + dB_{t} , \quad X_{0} = x
\end{equation}
where $b : \mathbb{R} \to \mathbb{R}$ is smooth and satisfies:
\begin{align}
\nonumber
|b(x) - b(y)| & \leq C |x-y|, \quad && x,y \in \mathbb{R} \\
\label{dissip_cond}
b'(x) & \leq L, \quad && x \in \mathbb{R}
\end{align}
where $C >0$, $L \in \mathbb{R}$ are some constants. By the classical theory of SDEs, for all $x \in \mathbb{R}$, there exists a unique continuous, square-integrable process $(X_{t}(x))_{t \geq 0}$ satisfying \eqref{one_dim_diff}. Actually, by the Lipschitz assumption on $b$, there even exists a bi-continuous process $(X_{t}(x))_{t \geq 0, x \in \mathbb{R}}$ such that, for all $x \in \mathbb{R}$, $(X_{t}(x))_{t \geq 0}$ solves (\ref{one_dim_diff}). 

Let $x \in \mathbb{R}$. Consider the  solution $(\eta_{t}(x))_{t \geq 0}$ to the variation equation obtained by formally differentiating  (\ref{one_dim_diff}) with respect to $x$:

\begin{equation*}
d\eta_{t}(x) = b'(X_{t}) \eta_{t}(x) dt , \quad \eta_{0}(x) = 1
\end{equation*}

Note that this is a (random) linear ODE with explicit solution given by:

\begin{equation*}
\eta_{t}(x) = \exp \left( \int_{0}^{t} b'(X_{s}) ds \right) 
\end{equation*}
 
It is easy to prove that, for all $t \geq 0$ and $x \in \mathbb{R}$, the map $y \to X_{t}(y)$ is a.s. differentiable at $x$ and:

\begin{equation} 
\label{dissip_derivative}
\frac{dX_{t}}{dx} \overset{a.s.}{=} \eta_{t}(x)
\end{equation}

\begin{rk}
Note that $\eta_{t}(x) > 0$ for all $t \geq 0$ and $x \in \mathbb{R}$. This reflects the fact that, for all $x \leq y$, by a comparison theorem for SDEs (see Theorem 3.7 in Chapter IX in \cite{revuz2013continuous}), one has $X_{t}(x) \leq X_{t} (y)$ .
\end{rk}

Recall that a Markovian semi-group $(P_{t})_{t \geq 0}$ on a Polish space $E$ is said to satisfy the strong Feller property if, for all $t > 0$ and $\varphi : E \to \mathbb{R}$ bounded and Borel, the function $P_{t} \varphi : E \to \mathbb{R}$ defined by:
\[ P_{t} \varphi (x) = \int \varphi(y) P_{t} (x,dy), \quad x \in \mathbb{R} \]
is continuous. 

The strong Feller property is very useful in the study of SDEs and SPDEs, namely for the proof of ergodicity (see, e.g., the monographs \cite{cerrai2001second}, \cite{da1996ergodicity} and \cite{zambotti2017random}, as well as the recent articles \cite{hairer2016strong} and \cite{tsatsoulis2016spectral}, for applications of the strong Feller property in the context of SPDEs).

Let $(P_{t})_{t \geq 0}$ be the Markovian semi-group associated to the SDE \eqref{sde}. We are interested in proving the strong Feller property for $(P_{t})_{t \geq 0}$. Note that, by assumption (\ref{dissip_cond}), $ \eta_{t}(x) \leq e^{Lt}$ for all $t \geq 0$ and $x \in \mathbb{R}$. Therefore, by (\ref{dissip_derivative}) and the dominated  convergence theorem, for all $\varphi : \mathbb{R} \to \mathbb{R}$ differentiable with a bounded derivative, one has:
\begin{equation*}
\frac{d}{dx}  \left( P_{t} \varphi \right) (x) = \frac{d}{dx} \mathbb{E} \left[ \varphi( X_{t}(x) ) \right] =  \mathbb{E} \left[ \varphi( X_{t}(x) )  \eta_{t}(x) \right]
\end{equation*}
As a consequence, for all $t \geq 0$, $P_{t}$ preserves the space $C^{1}_{b}(\mathbb{R})$ of bounded, continuously differentiable functions on $\mathbb{R}$ with a bounded derivative. It turns out that, actually, for all $t>0$, $P_{t}$ maps the space $ C_{b}(\mathbb{R})$ of bounded and  continuous functions into $ C^{1}_{b}(\mathbb{R})$. This is a consequence of the following, nowadays well-known, result:

\begin{thm}[Bismut-Elworthy-Li formula]
\label{belformula}
For all $T>0$ and $\varphi \in C_{b}(\mathbb{R})$, the function $P_{T} \varphi$ is differentiable and we have:
\begin{equation}
\frac{d}{dx} P_{T} \varphi (x) = \frac{1}{T} \mathbb{E} \left[ \varphi(X_{T}(x)) \int_{0}^{T} \eta_{s}(x) dB_{s} \right] 
\end{equation}
\end{thm}

\begin{proof}
See \cite{ELWORTHY1994252}, Theorem 2.1, or \cite{zambotti2017random}, Lemma 5.17 for a proof.
\end{proof}

\begin{cor}
\label{dissip_strong_feller}
The semi-group $(P_{t})_{t \geq 0}$ satisfies the strong Feller property and, for all $T>0$ and $\varphi : \mathbb{R} \to \mathbb{R}$ bounded and Borel, one has:
\begin{equation}
\label{sgpcontinuity}
\forall x,y \in \mathbb{R},\quad |P_{T} \varphi (x) - P_{T} \varphi (y)| \leq e^{L} \frac{||\varphi||_{\infty}}{\sqrt{T \wedge 1}} |x - y|,
\end{equation}
where $||\cdot||_{\infty}$ denotes the supremum norm.
\end{cor}

The following remark is crucial.

\begin{rk}
Inequality (\ref{sgpcontinuity}) involves only the dissipativity constant $L$, not the Lipschitz constant $C$. This makes the Bismut-Elworthy-Li formula very useful in the study of SPDEs with a dissipative drift.
\end{rk}  

\begin{proof}[Proof of Corollary \ref{dissip_strong_feller}] 
By approximation, it suffices to prove (\ref{sgpcontinuity}) for $\varphi \in C_{b}(\mathbb{R})$. For such a $\varphi$ and for all $T > 0$, by the Bismut-Elworthy-Li formula, one has:
\[ \left|  \frac{d}{dx} P_{T} \varphi (x) \right| \leq \frac{||\varphi||_{\infty}}{T} \mathbb{E} \left[ \left| \int_{0}^{T} \eta_{s}(x) dB_{s} \right| \right] \]
Remark that the process $(\eta_{t}(x))_{t \geq0}$ is locally bounded since it is dominated by $(e^{Lt})_{t \geq 0}$, so that the stochastic integral $\left( \int_{0}^{t} \eta_{s}(x) dB_{s} \right)_{t \geq 0}$ is an $L^{2}$ martingale. Hence using Jensen's inequality as well as It\^{o}'s isometry formula, we obtain:
\begin{align*}
\mathbb{E} \left[ | \int_{0}^{T} \eta_{s}(x) dB_{s} | \right] &\leq \sqrt{ \mathbb{E} \left[ \int_{0}^{T} \eta_{s}(x)^{2} ds | \right] } \\
 & \leq \sqrt{ \int_{0}^{T} e^{2Ls} \ ds } 
 \end{align*}
and the last quantity is bounded by  $\sqrt{e^{2L} T} = e^{L} \sqrt{T}$ for all $T \in (0,1]$. Therefore, we deduce that:
\[ \forall x \in \mathbb{R}, \quad \left|  \frac{d}{dx} P_{T} \varphi (x) \right| \leq e^{L} \frac{||\varphi||_{\infty}}{\sqrt{T}} \]
so that:
\[ \forall x, y \in \mathbb{R}, \quad \left| P_{T} \varphi (x) - P_{T} \varphi (y) \right| \leq e^{L} \frac{||\varphi||_{\infty}}{\sqrt{T}} |x-y| \]
for all $\varphi \in C_{b}(\mathbb{R})$ and $T \in (0,1]$. The case $T >1$ follows at once by using the semi-group property of $(P_{t})_{t \geq 0}$:
\begin{align*}
\left| P_{T} \varphi (x) - P_{T} \varphi (y) \right| & = \left| P_{1} \left( P_{T-1} \varphi \right) (x) - P_{1} \left( P_{T-1} \varphi \right) (y) \right| \\
& \leq e^{L} \frac{||P_{T-1} \varphi||_{\infty}}{\sqrt{1}} |x-y| \\
& \leq e^{L} ||\varphi||_{\infty} |x-y|
\end{align*}
The claim follows.
\end{proof}

\begin{rk}[A brief history of the Bismut-Elworthy-Li formula]
A particular form of this formula had originally been derived by J.M. Bismut in \cite{bismut1984large} using Malliavin calculus in the framework of the study of the logarithmic derivative of the fundamental solution of the heat equation on a compact manifold. In \cite{ELWORTHY1994252}, K.D. Elworthy and X.-M. Li generalized this formula to a large class of diffusion processes on manifolds with smooth coefficients, and gave also variants of this formula to higher-order derivatives.
\end{rk}

The key property allowing the analysis performed in this section is the dissipativity property (\ref{dissip_cond}). Without this property being true, one would not even expect the Bismut-Elworthy-Li formula to hold. However, in the sequel, we shall prove that results such as Theorem \ref{belformula} and Corollary \ref{dissip_strong_feller} above can also be obtained for a family of diffusions with a non-dissipative drift (informally $L= +\infty$) , namely for the Bessel processes of  dimension smaller than $1$.


\section{Bessel processes: notations and basic facts}
\label{basicsection}
 
In the sequel, for any subinterval $I$ of $\mathbb{R}_{+}$, $C(I)$ will denote the set of continuous functions $I \rightarrow \mathbb{R}$. We shall consider this set endowed with the topology of uniform convergence on compact sets, and will denote by $\mathcal{B}(C(I))$ the corresponding Borel $\sigma$-algebra.

Consider the canonical measurable space $(C(\mathbb{R}_{+}), \mathcal{B}(C(I)))$ endowed with the canonical filtration $(\mathcal{F}_{t})_{t \geq 0}$. Let $(B_{t})_{t \geq 0}$ be a standard linear $(\mathcal{F}_{t})_{t \geq 0}$-Brownian motion. For all $x \geq 0$ and $\delta \geq 0$, there exists a unique continuous, predictable, nonnegative process $(X^{\delta}_{t}(x))_{t \geq 0}$ satisfying:
\begin{equation}
\label{sde}
X_{t} = x^{2} + 2\int_{0}^{t} \sqrt{X_{s}} dB_{s} + \delta t.
\end{equation} 
$(X^{\delta}_{t}(x))_{t \geq 0}$ is a squared Bessel process of dimension $\delta$ started at $x^{2}$, and the process $\rho^{\delta}_{t}(x) := \sqrt{X^{\delta}_{t}(x)} $ is a $\delta$-dimensional Bessel process started at $x$. In the sequel, we will also write the latter process as $(\rho_{t}(x))_{t \geq 0}$, or $\rho$, when there is no risk of ambiguity. 

We recall the following monotonicity property of the family of Bessel processes:

\begin{lm}
\label{weakcomp}
For all couples $(\delta, \delta'),(x,x') \in \mathbb{R}_{+}$ such that $\delta \leq \delta'$ and $x \leq x'$, we have, a.s.:
\[ \forall t \geq 0, \qquad {\rho}^{\delta}_{t}(x) \leq {\rho}^{\delta'}_{t}(x'). \]
\end{lm}

\begin{proof}
By Theorem (3.7) in \cite{revuz2013continuous}, Section IX, applied to the equation (\ref{sde}), the following property holds a.s.:
\[ \forall t \geq 0, \qquad X^{\delta}_{t}(x) \leq X^{\delta'}_{t}(x'). \]
Taking the square root on both sides above, we deduce the result.
\end{proof}

For all $a \geq 0$, let $T_{a}(x)$ denote the $(\mathcal{F}_{t})_{t \geq 0}$ stopping time defined by:
\[ T_{a}(x) := \inf \{ t > 0, \rho_{t}(x) \leq a \} \]
(we shall also write $T_{a}$). We recall the following fact, (see e.g. Proposition 3.6 of \cite{zambotti2017random}):

\begin{prop}
The following dichotomy holds:
\begin{itemize}
\item $T_{0}(x) = + \infty$ a.s., if $\delta \geq 2$,
\item $T_{0}(x) < + \infty$ a.s., if $0 \leq \delta < 2$.  
\end{itemize}
\end{prop}

Applying It\^{o}'s lemma to $\rho_{t} = \sqrt{X^{\delta}_{t}(x)}$, we see that $\rho$ satisfies the following relation on the interval $[0, T_{0})$:
\begin{equation}
\label{sdebessel}
 \forall t \in [0, T_{0}), \qquad \rho_{t} = x + \frac{\delta - 1}{2} \int_{0}^{t} \frac{ds}{\rho_{s}} + B_{t} .
\end{equation} 


\section{Derivative in space of the Bessel semi-group}

Let $\delta \geq 0$. We denote by $P^{\delta}_{x}$ the law, on $(C(\mathbb{R}_{+}), \mathcal{B}(C(\mathbb{R}_{+})))$, of the $\delta$-dimensional Bessel process started at x, and we write $E^{\delta}_{x}$ for the corresponding expectation operator. We also denote by $(P^{\delta}_{t})_{t \geq 0}$ the family of transition kernels associated with the $\delta$-dimensional Bessel process, defined by
\[ P^{\delta}_{t} F(x) := E^{\delta}_{x}(F(\rho_{t})) \]
for all $t \geq 0 $ and all $F : \mathbb{R}_{+} \to \mathbb{R}$ bounded and Borel. The aim of this section is to prove the following:

\begin{thm}
\label{thmana}
For all $T>0$ and all $F : \mathbb{R}_{+} \to \mathbb{R}$ bounded and Borel, the function $x \to P^{\delta}_{t} F (x)$ is differentiable on $\mathbb{R}_{+}$, and for all $x \geq 0$:
\begin{equation}
\label{bel2}
\frac{d}{dx}  P^{\delta}_{T} F (x) = \frac{x}{T} \left( P^{\delta +2}_{T} F (x) - P^{\delta}_{T} F (x) \right). 
\end{equation}
In particular, the function $x \to P^{\delta}_{t} F (x)$ satisfies the Neumann boundary condition at $0$:
\[ \frac{d}{dx} P^{\delta}_{T} F (x) \Big\rvert_{x=0} = 0 . \]
\end{thm}

\begin{rk} 
By Theorem \ref{thmana}, the derivative of the function $ x \mapsto P^{\delta}_{T} F (x)$ is a smooth function of $P^{\delta +2}_{T} F (x)$ and $P^{\delta}_{T} F (x)$. Hence, reasoning by induction, we deduce that the function $ x \mapsto P^{\delta}_{T} F (x)$ is actually smooth on $\mathbb{R}_{+}$.
\end{rk}

\begin{proof}
The proof we propose here relies on the explicit formula for the transition semi-group of the Bessel processes. We first treat the case $\delta > 0$.

Given $\delta > 0$, let $\nu := \frac{\delta}{2} - 1$, and denote by $I_{\nu}$ the modified Bessel function of index $\nu$. We have (see, e.g., Chap. XI.1 in \cite{revuz2013continuous}) :
\[ P^{\delta}_{t}F(x) = \int_{0}^{\infty} p^{\delta}_{T}(x,y) F(y) dy \]
where, for all $y \geq 0$:
\begin{align*}
 p^{\delta}_{t}(x,y) &= \frac{1}{T} \left( \frac{y}{x} \right)^{\nu} y  \exp \left(-\frac{x^{2}+y^{2}}{2T} \right) I_{\nu} \left( \frac{xy}{T} \right) , \quad \text{if} \quad x > 0, \\
p^{\delta}_{t}(0,y) &= \frac{2^{-\nu} T^{-(\nu+1)}}{ \Gamma(\nu +1)} y^{2 \nu +1} \exp \left(-\frac{y^{2}}{2T} \right)
\end{align*}
where $\Gamma$ denotes the gamma function. By the power series expansion of the function $I_{\nu}$ we have, for all $x,y \geq 0$:
\begin{equation} 
\label{sgpdensity}
p^{\delta}_{T}(x,y) = \frac{1}{T} \exp \left( -\frac{x^{2}+y^{2}}{2T} \right) \tilde{p}^{\delta}_{T}(x,y)
\end{equation}
with:
\[ \tilde{p}^{\delta}_{T}(x,y) := \sum_{k=0}^{\infty} \frac{y^{2k+2\nu+1} \ x^{2k} \ (1/2T)^{2k+\nu}}{k! \ \Gamma(k + \nu + 1)}. \] 
Note that $\tilde{p}^{\delta}_{T}(x,y)$ is the sum of a series with infinite radius of convergence in $x$, hence we can compute its derivative by differentiating under the sum. We have:
\begin{align*}
\label{xderiv}
\frac{\partial}{\partial x} \tilde{p}^{\delta}_{T}(x,y) &= \frac{\partial}{\partial x} \left(  \sum_{k=0}^{\infty} \frac{ x^{2k} \ y^{2k+2\nu+1} \ (1/2T)^{2k+\nu}}{k! \ \Gamma(k + \nu + 1)} \right) \\
&= \sum_{k=0}^{\infty} \frac{ 2k \ x^{2k-1} \ y^{2k+2\nu+1} \ (1/2T)^{2k+\nu}}{k! \ \Gamma(k + \nu + 1)} \\
&= \frac{x}{T} \ \sum_{k=1}^{\infty} \frac{ x^{2k-2} \ y^{2k+2\nu+1} \ (1/2T)^{2k+\nu -1}}{(k-1)! \ \Gamma(k + \nu + 1)}.
\end{align*}
Hence, performing the change of variable $j = k-1$, and remarking that $\nu + 1 = \frac{\delta +2}{2} -1$, we obtain: 
\begin{align*}
\frac{\partial}{\partial x} \tilde{p}^{\delta}_{T}(x,y) &= \frac{x}{T} \ \sum_{j=0}^{\infty} \frac{x^{2(j+1)-2} \ y^{2(j+1)+2\nu+1} \ (1/2T)^{2(j+1)+\nu -1}}{j! \ \Gamma((j+1) + \nu + 1)} \\
&= \frac{x}{T} \ \sum_{j=0}^{\infty} \frac{x^{2j} \ y^{2j+2(\nu+1)+1} \ (1/2T)^{2j+(\nu+1)}}{j! \ \Gamma(j + (\nu+1) + 1)} \\
&= \frac{x}{T}\ \tilde{p}^{\delta +2}_{T}(x,y).
\end{align*} 
As a consequence, differentiating equality (\ref{sgpdensity}) with respect to x, we obtain:
\begin{align*}
 \frac{\partial}{\partial x} p^{\delta}_{T}(x,y) &= \left( - \frac{x}{T} \ \tilde{p}^{\delta}_{T}(x,y) + \frac{\partial}{\partial x}\tilde{p}^{\delta}_{T}(x,y) \right) \frac{1}{T} \exp \left( -\frac{x^{2}+y^{2}}{2T} \right) \\
 &= \frac{x}{T} \ \left( - p^{\delta}_{T}(x,y) + p^{\delta+2}_{T}(x,y) \right).
\end{align*}
Hence, we deduce that the function $x \mapsto P^{\delta}_{T}F(x)$ is differentiable, with a derivative given by (\ref{bel2}). 

Now suppose that $\delta = 0$. We have, for all $x \geq 0$:
\begin{equation}
\label{sgp_zero}
 P^{0}_{T} F (x) = \exp \left(-\frac{x^{2}}{2T} \right) F(0) + \int_{0}^{\infty} p_{T}(x,y) F(y) dy 
\end{equation} 
where, for all $y \geq 0$:
\[ p_{T}(x,y) =  \frac{1}{T} \exp \left( -\frac{x^{2}+y^{2}}{2T} \right)  \tilde{p}_{T}(x,y) \]
with:
\[ \tilde{p}_{T}(x,y) := x \ I_{1} \left( \frac{xy}{T} \right) = \sum_{k=0}^{\infty} \frac{x^{2k+2} \ (y/2T)^{2k+1}}{k! (k + 1) !}. \]
Here again, we can differentiate the sum term by term, so that, for all $x, y \geq 0$:
\begin{align*}
\frac{\partial}{\partial x} \tilde{p}_{T}(x,y) & = \frac{x}{T} \ \sum_{k=0}^{\infty} \frac{x^{2k} y^{2k+1} (1/2T)^{2k}}{k!^{2}} \\
& = \frac{x}{T} \ \tilde{p}^{2}_{T}(x,y).
\end{align*}
Therefore, for all $x,y \geq 0$, we have:
\begin{align*} 
\frac{\partial}{\partial x} p_{T}(x,y) &=  \left( - \frac{x}{T} \tilde{p}_{T}(x,y) +  \frac{\partial}{\partial x} \tilde{p}_{T}(x,y) \right) \frac{1}{T} \exp \left( -\frac{x^{2}+y^{2}}{2T} \right)  \\
&=  \frac{x}{T} \left( - \tilde{p}_{T}(x,y) +  \tilde{p}^{2}_{T}(x,y) \right) \frac{1}{T} \exp \left( -\frac{x^{2}+y^{2}}{2T} \right)  \\
&=  \frac{x}{T} \left( - p_{T}(x,y) +  p^{2}_{T}(x,y) \right)
\end{align*}
Hence, differentiating (\ref{sgp_zero}) with respect to $x$, and using the dominated convergence theorem to differentiate inside the integral, we obtain:
\begin{align*}
 \frac{\partial}{\partial x} P^{0}_{T} F (x) &=  - \frac{x}{T} \exp(-\frac{x^{2}}{2T}) F(0)  + \frac{x}{T} \int_{0}^{\infty} \left( - p_{T}(x,y) + p^{2}_{T}(x,y) \right) F(y) dy \\
 &= \frac{x}{T} \ \left( - P^{0}_{T} F (x)  + P^{2}_{T} F (x)  \right),
\end{align*}
which yields the claim.
\end{proof}

\begin{rk}
Formula (\ref{bel2}) can also be derived using the Laplace transform of the one-dimensional marginals of the squared Bessel processes. Indeed, denote by $(Q^{\delta}_{t})_{t \geq 0}$ the family of transition kernels of the $\delta$-dimensional squared Bessel process. Then for all $\delta \geq 0$, $x \geq 0$, $T > 0$, and all function $f$ of the form $f(x) = \exp(-\lambda x)$ with $\lambda \geq 0$, one has:
\[ Q^{\delta}_{T} f(x) = \exp \left( - \frac{\lambda x}{1+2\lambda T} \right) (1+2 \lambda T)^{-\delta/2} 
\]
(see \cite{revuz2013continuous}, Chapter XI, Cor. (1.3)).
For such test functions $f$, we check at once that the following equality holds:
\[ \frac{d}{dx} Q^{\delta}_{T}f(x) = \frac{1}{2T} \left( Q^{\delta + 2}_{T}f(x)  - Q^{\delta}_{T}f(x) \right). \]
By linearity and by the Stone-Weierstrass theorem, we deduce that this equality holds for all bounded, continuous functions $f$ . Then an approximation argument enables to deduce the equality for all functions $f: \mathbb{R}_{+} \rightarrow \mathbb{R} $ Borel and bounded. Finally, remarking that for all bounded Borel function $F$ on $\mathbb{R}_{+}$ we have
\[ P^{\delta}_{T} F(x) = Q^{\delta}_{T} f (x^{2}) \]
with $f(x) := F(\sqrt{x})$, we deduce that:
\begin{align*}
\frac{d}{dx} P^{\delta}_{T}F(x) & =  2x \ \frac{d}{dx}(Q^{\delta}_{T} f) (x^{2}) \\
& = \frac{x}{T} \left(Q^{\delta + 2}_{T}f(x^{2})  - Q^{\delta}_{T}f(x^{2}) \right) \\
&= \frac{x}{T} \left(P^{\delta + 2}_{T}F(x)  - P^{\delta}_{T}F(x) \right)
\end{align*}
which yields the equality (\ref{bel2}).
\end{rk}

\begin{cor}
The semi-group $(P^{\delta}_{t})_{t \geq 0}$ has the strong Feller property. More precisely, for all $T > 0$, $R>0$, $x,y  \in [0,R]$ and  $F: \mathbb{R}_{+} \to \mathbb{R}$ bounded and Borel, we have:
\begin{equation}
\label{feller}
 | P^{\delta}_{T} F (x) - P^{\delta}_{T} F (y)| \leq \frac{ 2 R ||F||_{\infty}}{T} |y-x|. 
\end{equation}
\end{cor}

\begin{proof}
By Theorem \ref{thmana}, for all $x,y \in [0,R]$ such that $x \leq y$, we have:
\begin{align*}
|P^{\delta}_{t} F (x) - P^{\delta}_{t} F (y)| &= \left| \int_{x}^{y}  \frac{u}{T} \left(  P^{\delta+2}_{T} F(u) -  P^{\delta}_{T} F(u)  \right) du \right| \\
&\leq \frac{2 ||F||_{\infty}}{T}  \int_{x}^{y} u \ du \\
&\leq \frac{ 2 R ||F||_{\infty}}{T} |y-x| .
\end{align*}
\end{proof}

\begin{rk}
The bound (\ref{feller}) is in $1/T$, which is not very satisfactory for $T$ small. However, in the sequel, we will improve this bound by getting a better exponent on $T$, at least for $\delta \geq 2(\sqrt{2} -1)$ (see inequality (\ref{betterbound}) below).
\end{rk} 


\section{Differentiability of the flow}
\label{diffflow}

In the following, we are interested in finding a probabilistic interpretation of Thm \ref{thmana}, in terms of the Bismut-Elworthy-Li formula. To do so we study, for all $\delta \geq 0$, and all couple $(t,x) \in \mathbb{R}_{+} \times \mathbb{R}^{*}_{+}$, the differentiability at $x$ of the function:
\begin{align*} 
\rho_ {t} \colon & \mathbb{R}_{+} \to \mathbb{R}_{+} \\
                               & y \mapsto \rho^{\delta}_{t}(y).
\end{align*}
In this endeavour, we first need to choose an appropriate modification of the process $(\rho_{t}(x))_{t \geq 0, x > 0}$. We have the following result:

\begin{prop}
\label{modif}
Let $\delta \geq 0$ be fixed. There exists a modification $(\tilde{\rho}^{\delta}_{t}(x))_{x , t \geq 0}$ of the process $(\rho^{\delta}_{t}(x))_{x, t \geq 0}$ such that, a.s., for all $x, x' \in \mathbb{R}_{+}$ with $x \leq x' $ ,  we have:
\begin{equation}
\label{comp}
\forall t \geq 0, \qquad {\tilde{\rho}}^{\delta}_{t}(x) \leq {\tilde{\rho}}^{\delta}_{t}(x') .
\end{equation}  
\end{prop} 

\begin{proof}
For all $q,q' \in \mathbb{Q}_{+}$, such that $q \leq q'$, by Lemma \ref{weakcomp}, the following property holds a.s.:
\[ \forall t \geq 0, \qquad \rho^{\delta}_{t}(q) \leq \rho^{\delta}_{t}(q'). \]
For all $x \in \mathbb{R}_{+} $, we define the process $\tilde{\rho}^{\delta}(x)$ by:
\[ \forall t \geq 0, \qquad \tilde{\rho}^{\delta}_{t}(x) := \underset{q \in \mathbb{Q}_{+}, q \geq x}{\inf} \rho^{\delta}_{t}(q). \]
Then $(\tilde{\rho}^{\delta}_{t}(x))_{x, t \geq 0}$ yields a modification of the process $(\rho^{\delta}_{t}(x))_{x, t \geq 0}$ with the requested property.
\end{proof}  

\begin{rk}
\label{ascont}
We may not have, almost-surely, joint continuity of all the functions $t \mapsto \tilde{\rho}_{t}(x)$, $x \geq 0$. Note however that, by definition, for all $x\geq0, x \in \mathbb{Q}$, we have a.s.:
\[ \forall t \geq 0, \qquad \tilde{\rho}_{t}(x) = \rho_{t}(x), \]
so that, a.s., $t \mapsto {\tilde{\rho}}_{t}(x)$ is continuous and satisfies: 
\[ \forall t \in [0, T_{0}(x)), \qquad \tilde{\rho}_{t}(x) = x + \frac{\delta - 1}{2} \int_{0}^{t} \frac{ds}{\tilde{\rho}_{s}} + B_{t} . \]
As a consequence, by countability of $\mathbb{Q}$, there exists an almost sure event $\mathcal{A} \in \mathcal{F}$ on which, for all $x \in \mathbb{Q}_{+}$, the function $ t \mapsto {\tilde{\rho}}_{t}(x)$ is continuous and satisfies:
\[ \forall t \in [0, T_{0}(x)), \qquad \tilde{\rho}_{t}(x) = x + \frac{\delta - 1}{2} \int_{0}^{t} \frac{ds}{\tilde{\rho}_{s}} + B_{t}. \]
\end{rk}                    

In this section, as well as the Appendix, we always work with the modification $\tilde{\rho}$.  Similarly, we work with :
\[ \tilde{T}_{0}(x) := \inf \{ t > 0, \tilde{\rho}^{\delta}_{t}(x) = 0 \} \]
instead of $T_{0}(x)$, for all $\delta, x \geq 0$. We will write again $\rho$ and $T_{0}$ instead of $ \tilde{\rho}$ and $\tilde{T_{0}}$. Note that, a.s., the function $x \mapsto T_{0}(x)$ is non-decreasing on $\mathbb{R}_{+}$. 

\begin{prop}
\label{derflow}
Let $\delta \geq 0$, $t > 0$ and $x >0 $. Then, a.s., the function $\rho_{t}$ is differentiable at $x$, and its derivative there is given by:
\begin{equation}\label{indic} 
\frac{d\rho_{t}(y)}{dy}\Big\rvert_{y=x} \overset{a.s.}{=} \eta_{t}(x) := \mathbf{1}_{t < T_{0}(x)} \exp \left( \frac{1-\delta}{2} \int_{0}^{t}\frac{ds}{\rho_{s}(x)^{2}} \right). 
\end{equation}
\end{prop}

The proof of this proposition is quite technical. Since, moreover, the result will not be necessary in the sequel, we prefer to postpone the proof to the Appendix of the article.

\begin{rk}
In particular, when $\delta =1$, the above formula reduces to:
\begin{equation}
\label{discontder}
 \frac{d\rho_{t}(y)}{dy}\Big\rvert_{y=x} \overset{a.s.}{=} \mathbf{1}_{t < T_{0}(x)}
 \end{equation}
a formula which was already well-known (see e.g. \cite{arnaudon2017reflected}, Lemma A.1). 
\end{rk}

\begin{rk}
Note that the indicator function $\mathbf{1}_{t < T_{0}(x)}$ in the right-hand side of 
(\ref{indic}) is related to the behavior of the Bessel process at the boundary ${0}$. It is reminiscent of Theorem 1 in \cite{deuschel2005bismut} , where a similar indicator function appears in the expression of the spatial derivative of the flow of vector-valued solutions to SDEs with reflection. 
\end{rk} 

\begin{rk}\label{rk5}
Proposition \ref{derflow} shows that, for all $t , x > 0$, the function $\rho_{t}$ is almost-surely differentiable at $x$. We may, however, ask if, a.s., the function $\rho_{t}$ is differentiable on the whole of $\mathbb{R}^{*}_{+}$. The case where $\delta > 1$ was treated in detail in \cite{vostrikova2009regularity}, where it was shown that, a.s., for all $t \geq 0$ the function  $x \mapsto \rho_{t}(x)$ is differentiable on $\mathbb{R}^{*}_{+}$, and that the derivative $\frac{d\rho_{t}(x)}{dx}$ is continuous in $ (t,x) \in \mathbb{R}_{+} \times \mathbb{R}^{*}_{+}$. However, as $\delta$ gets smaller than $1$, the regularity of the process $(\rho_{t}(x))_{t \geq 0, x > 0}$ becomes much worse. Note that $\delta =1$ corresponds to the case of the flow of reflected Brownian motion on the half-line; in that case the flow is no longer continuously differentiable as suggested by (\ref{discontder}). Many works have been carried out on the study of the flow of reflected Brownian motion on domains in higher dimension (see e.g. \cite{burdzy2009differentiability} and \cite{varadhan1985brownian}) or on manifolds with boundary (see e.g. \cite{arnaudon2017reflected}). By contrast, the regularity of Bessel flows of dimension $\delta < 1$ seems to be a very open problem. 

In the remainder of the article, however, we shall not need any regularity results on the Bessel flow. Instead, for all fixed $x > 0$, we shall study the process $(\eta_{t}(x))_{t \geq 0}$ defined above as a process in itself.
\end{rk}


\section{Properties of $\eta$}

In the sequel, for all $x \geq 0$, we shall consider the process $(\eta_{t}(x))_{t \geq 0}$ defined as  above:
\begin{equation}
\label{defeta}
 \eta_{t}(x) := \mathbf{1}_{t < T_{0}(x)} \exp \left( \frac{1-\delta}{2} \int_{0}^{t}\frac{ds}{\rho_{s}(x)^{2}} \right).
 \end{equation}
When there is no ambiguity we shall drop the $x$ from our notation and denote this process by $\eta$. 


\subsection{Regularity of the sample paths of $\eta$}

We are interested in the continuity property of the process $\eta$. It turns out that, as $\delta$ decreases, $\eta$ becomes more and more singular, as shown by the following result. 

\begin{prop} 
\label{continuityeta}
If $\delta > 1$, then a.s. $\eta$ is bounded and continuous on $\mathbb{R}_{+}$. \\
If $\delta =1$, then a.s. $\eta$ is constant on $[0,T_{0})$ and $[T_{0},+\infty)$, but has a discontinuity at $T_{0}$. \\
If $\delta \in [0,1)$, then a.s. $\eta$ is continuous away from $T_{0}$, but it diverges to $+ \infty$ as $t \uparrow T_{0}$.
\end{prop}

\begin{proof}
When $\delta \geq 2$, $T_{0} = \infty$ almost-surely, so that, by (\ref{defeta}), the following equality of processes holds:
\[ \eta_{t} = \exp \left( \frac{1-\delta}{2} \int_{0}^{t}\frac{ds}{\rho_{s}(x)^{2}} \right). \]
Hence, a.s., $\eta$ takes values in $[0,1]$ and is continuous on $\mathbb{R}_{+}$. 
To  treat the case $\delta < 2$ we need a lemma:

\begin{lm}
\label{integral}
Let $\delta < 2$ and $ x > 0$. Then the integral:
\[ \int_{0}^{T_{0}} \frac{ds}{(\rho_{s}(x))^{2}} \]
is infinite a.s.
\end{lm}

We admit this result for the moment. Then, when $\delta \in (1, 2)$, $\eta$ takes values in $[0,1]$, is continuous away from $T_{0}$ and, almost-surely, as $t  \uparrow T_{0}$:
\[ \eta_{t} = \exp \left( \frac{1-\delta}{2} \int_{0}^{t}\frac{ds}{\rho_{s}(x)^{2}} \right) \longrightarrow 0.\]
Since, $\eta_{t} = 0$ for all $t \geq T_{0}$, $\eta$ is continuous and the claim follows.
When $\delta =1$,
\[  \eta_{t}(x) := \mathbf{1}_{t < T_{0}(x)} \]
so the claim follows at once.  
Finally, if $\delta \in [0,1)$, then $\eta$ is continuous away from $T_{0}$, but by the above lemma, a.s., as $t \uparrow T_{0}$:
\[ \eta_{t} = \exp \left( \frac{1-\delta}{2} \int_{0}^{t}\frac{ds}{\rho_{s}(x)^{2}} \right) \longrightarrow + \infty\]
so the claim follows.
\end{proof}

We now prove Lemma \ref{integral}

\begin{proof}[Proof of Lemma \ref{integral}]
The proof is in two steps. In a first step we prove the lemma when $\rho$ is replaced with a Brownian motion started at some positive point, and in a second step we invoke a representation theorem of Bessel processes as time-changes of some power of the Brownian motion to conclude. 

\textit{First step}:
Let $(\beta_{t})$ be a Brownian motion started from some $y > 0$, and let $T_{0}$ denote its hitting time of the origin. Then the integral :
\[   \int_{0}^{T_{0}} \frac{ds}{(\beta_{s}(y))^{2}} \]
is a.s. infinite. Indeed, denote by $h: [0,\infty) \rightarrow \mathbb{R}_{+}$ the function given by:
\[ h(t) := \begin{cases} \sqrt{t |\log(1/t)|}, & \  \text{if} \ t > 0, \\
                                        0 ,                          &   \  \text{if} \ t = 0. 
              \end{cases} \]                          
Let $A > 0$.  By Levy's modulus of continuity (see Theorem (2.7), Chapter I, in \cite{revuz2013continuous}), there exists a $\kappa > 0$, such that the event
\[ \mathcal{M} := \{ \ \forall s,t \in [0, 1], \quad |\beta_{t} - \beta_{s}| \leq \kappa \ h(|s-t|) \ \} \]
has probability one. Therefore, by scale invariance of Brownian motion, setting $\kappa_{A} := \sqrt{A} \kappa$, one deduces that the event
\[ \mathcal{M}_{A} := \{ \ \forall s,t \in [0, A], \quad |\beta_{t} - \beta_{s}| \leq \kappa_{A} \ h(|s-t|) \ \} \]
also has probability one. Moreover, under the event $\{T_{0} < A \} \cap \mathcal{M}_{A}$, we have, for small $h >0$.
\[ \beta_{T_{0}-h}^{2} = |\beta_{T_{0} - h} - \beta_{T_{0}}|^{2} \leq {\kappa_{A}}^{2} \ h \log(1/h). \]
Since $\frac{1}{h \log(1/h)}$ is not integrable as $ h \to 0^{+}$, we deduce that, under the event \linebreak $\{T_{0} < A \} \cap \mathcal{M}_{A}$, we have $\int_{0}^{T_{0}} \frac{ds}{(\beta_{s})^{2}} = + \infty$. Therefore:
\begin{align*}
\mathbb{P} [ T_{0} < A] = \mathbb{P} [ \{ T_{0} < A \} \cap \mathcal{M}_{A} ]  \leq \mathbb{P} \left(  \int_{0}^{T_{0}} \frac{ds}{(\beta_{s})^{2}} = + \infty \right).
\end{align*}
Since $T_{0} < + \infty $ a.s., we have $\underset{A \to \infty}{\lim} \mathbb{P} [ T_{0} < A]  = 1$. Hence, letting $A \to \infty$ in the above, we deduce that:
\[  \mathbb{P} \left( \int_{0}^{T_{0}} \frac{ds}{(\beta_{s})^{2}} = + \infty \right) = 1 \]
as claimed. 

\textit{Second step}:
Now consider the original Bessel process $(\rho_{t}(x))_{t \geq 0}$. Suppose that $\delta \in (0,2)$. Then, by Thm 3.5 in \cite{zambotti2017random}, the process  $(\rho_{t}(x))_{t \geq 0}$ is equal in law to $(|\beta_{\gamma(t)}|^{\frac{1}{2-\delta}})_{t \geq 0}$, where $\beta$ is a Brownian motion started from $y := x^{2-\delta}$, and $\gamma : \mathbb{R}_{+} \rightarrow \mathbb{R}_{+}$ is the inverse of the increasing function $A: \mathbb{R}_{+} \rightarrow \mathbb{R}_{+}$ given by:
\[\forall u \geq 0, \qquad A(u) = \frac{1}{(2-\delta)^{2}} \int_{0}^{u} |\beta_{s}|^{\frac{2(\delta-1)}{2 - \delta}} ds . \]
Therefore, denoting by $T^{\beta}_{0}$ the hitting time of $0$ by the Brownian motion $\beta$,  we have:
\begin{align*}
\int_{0}^{T_{0}} \frac{ds}{(\rho_{s}(x))^{2}} & \overset{(d)}{=} \int_{0}^{A(T^{\beta}_{0})}  \frac{ds}{|\beta_{\gamma(s)}|^{\frac{2}{2-\delta}}} \\
& =  \int_{0}^{T^{\beta}_{0}}  \frac{1}{|\beta_{u}|^{\frac{2}{2-\delta}}} \frac{1}{(2-\delta)^{2}}  |\beta_{u}|^{\frac{2(\delta-1)}{2 - \delta}} du \\
& =   \frac{1}{(2-\delta)^{2}} \int_{0}^{T^{\beta}_{0}}  \frac{du}{{\beta_{u}}^{2}}
\end{align*}
where we have used the change of variable $u=\gamma(s)$ to get from the first line to the second one.  By the first step, the last integral is infinite a.s., so the claim follows. 

There still remains to treat the case $\delta =0$. By Thm 3.5 in \cite{zambotti2017random}, in that case, the process  $(\rho_{t}(x))_{t \geq 0}$ is equal in law to $\left( \left( \beta_{\gamma(t) \wedge T^{\beta}_{0}} \right)^{1/2} \right)_{t \geq 0}$, where $\beta$ is a Brownian motion started from $y := x^2$, $T^{\beta}_{0}$ is its hitting time of $0$ and $\gamma : \mathbb{R}_{+} \rightarrow \mathbb{R}_{+}$ is the inverse of the increasing function $A: \mathbb{R}_{+} \rightarrow \mathbb{R}_{+}$ given by:
\[\forall u \geq 0, \qquad A(u) = \frac{1}{4} \int_{0}^{u \wedge T^{\beta}_{0}} \beta_{s}^{-1} ds . \]
Then, the same computations as above yield the equality in law:
\begin{equation*}
\int_{0}^{T_{0}} \frac{ds}{(\rho_{s}(x))^{2}} \overset{(d)}{=} \frac{1}{4} \int_{0}^{T^{\beta}_{0}}  \frac{du}{{\beta_{u}}^{2}}
\end{equation*}
so the result follows as well.
\end{proof}


\subsection{Study of a martingale related to $\eta$}

Let $\delta \in [0,2)$ and $x >0$ be fixed. In the previous section,  we have shown that, a.s. :
\[ \int_{0}^{t} \frac{ds}{\rho_{s}(x)^{2}} \underset{t \to T_{0}(x)}{\longrightarrow} + \infty \]
As a consequence, for $\delta \in [0,1)$, a.s., the modification $\eta_{t}$ of the derivative at $x$ of the stochastic flow $\rho_{t}$ diverges at $T_{0}(x)$:
\[ \eta_{t}(x) = \mathbf{1}_{t < T_{0}(x)} \exp \left( \frac{1-\delta}{2} \int_{0}^{t} \frac{ds}{\rho_{s}(x)^{2}} \right) \underset{ t \uparrow T_{0}(x)}{\longrightarrow} + \infty . \]
However, since $\rho_{t}(x) \longrightarrow 0$ as $t \to T_{0}(x)$, this does not exclude the possibility that the product $\rho_{t}(x) \eta_{t}(x)$ converges as $t \to T_{0}(x)$. This motivates to study the process :
\begin{equation}
\label{martingale}
 D_{t} := \rho_{t}(x) \eta_{t}(x) = \mathbf{1}_{t < T_{0}(x)} \rho_{t}(x) \exp \left( \frac{1-\delta}{2} \int_{0}^{t} \frac{ds}{\rho_{s}(x)^{2}} \right).
\end{equation} 
As a matter of fact, we will show that $(D_{t})_{t \geq 0}$ is an $L^{p}$ continuous martingale for some $p \geq 1$.

\begin{rk}
The process $(D_{t})_{t \geq 0}$ appears as (one half times) the derivative of the stochastic flow associated with the squared Bessel process $X_{t}(x) = (\rho_{t}(x))^{2}$. Indeed, by applying formally the chain rule, we have, for all $t \geq 0$ and $x > 0$:
\[ \frac{d X_{t}(x)}{dx}  = 2 \rho_{t}(x)  \eta_{t}(x). \]
\end{rk}


\subsection{Continuity of $(D_{t})_{t \geq 0}$}

In this subsection we show that the process $(D_{t})_{t \geq 0}$ has a.s. continuous sample paths. By the expression (\ref{martingale}), continuity holds as soon as $T_{0}(x) = \infty$ a.s., i.e. as soon as $\delta \geq 2$. On the other hand, if $\delta \in [0,2)$ it suffices to prove that, a.s., $D_{t} \to 0$ as $t \uparrow T_{0}(x)$. This is the content of the following proposition.

\begin{prop}
\label{cv}
For all $\delta \in [0,2)$ and $x >0$, with probability one:
\[ \rho_{t} (x) \exp \left( \frac{1-\delta}{2} \int_{0}^{t} \frac{ds}{\rho_{s}(x)^{2}} \right) \underset{ t \to T_{0}(x)}{\longrightarrow} 0 . \]
\end{prop}

\begin{proof}
If $\delta \in [1,2)$, then $\exp \left( \frac{1-\delta}{2} \int_{0}^{t} \frac{ds}{\rho_{s}(x)^{2}} \right) \leq 1$ for all $t \geq 0$. Since $\rho_{t} \longrightarrow 0$ as $t \to T_{0}(x)$, the claim follows at once. 

If $\delta \in [0,1)$, on the other hand, $\exp \left( \frac{1-\delta}{2} \int_{0}^{t} \frac{ds}{\rho_{s}(x)^{2}} \right) \underset{t \uparrow T_{0}(x)}{ \xrightarrow{\hspace*{0.9 cm}}} + \infty$ whereas $\rho_{t} \underset{t \to T_{0}(x)}{\xrightarrow{\hspace*{0.9 cm}}} 0$ so a finer analysis is needed. We have:
\[ \log \left[ \frac{\rho_{t}}{x}\exp \left( \frac{1-\delta}{2} \int_{0}^{t} \frac{ds}{\rho_{s}(x)^{2}} \right) \right] = \log{\frac{\rho_{t}}{x}} + \frac{1-\delta}{2} \int_{0}^{t} \frac{ds}{\rho_{s}^{2}} \]
Now, recall that a.s., for all $t < T_{0}$, we have:
\[  \rho_{t} = x + \frac{\delta-1}{2}\int_{0}^{t} \frac{ds}{\rho_{s}} + B_{t} \]
Hence, defining for all integer $ n \geq 1$ the $(\mathcal{F}_{t})_{t \geq 0}$-stopping time $\tau_{n}$ as:
\[\tau_{n} := \inf \{ t > 0, \ \rho_{t} \leq 1/n \} \wedge n, \]
we have:
\[  \rho_{t \wedge \tau_{n}} = x + \frac{\delta-1}{2}\int_{0}^{t \wedge \tau_{n}} \frac{ds}{\rho_{s}} + B_{t \wedge \tau_{n}}. \]
Hence, by It\^{o}'s lemma, we deduce that:
\[\log{\frac{\rho_{t \wedge \tau_{n}}}{x}} = \frac{\delta-1}{2} \int_{0}^{t \wedge \tau_{n}} \frac{ds}{\rho_{s}^{2}} + \int_{0}^{t \wedge \tau_{n}} \frac{dB_{s}}{\rho_{s}} - \frac{1}{2} \int_{0}^{t \wedge \tau_{n}}  \frac{ds}{\rho_{s}^{2}} \]
so that:
\begin{equation}
\label{logofmart}
 \log{\frac{\rho_{t \wedge \tau_{n}}}{x}} + \frac{1-\delta}{2} \int_{0}^{t \wedge \tau_{n}} \frac{ds}{\rho_{s}^{2}} = \int_{0}^{t \wedge \tau_{n}} \frac{dB_{s}}{\rho_{s}} - \frac{1}{2} \int_{0}^{t \wedge \tau_{n}}  \frac{ds}{\rho_{s}^{2}}. 
\end{equation}
Consider now the random time change: 
\begin{alignat*}{2} 
A \colon [0,T_{0}) &\to \mathbb{R}_{+} \\
                 t               &\mapsto     A_{t} := \int_{0}^{t} \frac{ds}{\rho_{s}^{2}}. 
\end{alignat*}                 
Note that $A$ is differentiable with strictly positive derivative. Moreover, since $A_{t} \underset{t \to T_{0}}{\longrightarrow} + \infty$ a.s. by Lemma \ref{integral}, we deduce that $A$ is a.s. onto. Hence, a.s., $A$ is a diffeomorphism $[0,T_{0}) \rightarrow \mathbb{R}_{+}$, the inverse of which we denote by 
\begin{align*} 
C \colon   \mathbb{R}_{+} & \rightarrow [0,T_{0}) \\
                  u & \mapsto C_{u} .     
\end{align*}                 
Let $\beta_{u}:= \int_{0}^{C_{u}} \frac{dB_{r}}{\rho_{r}},  u \geq 0$. Then $\beta$ is a local martingale started at $0$ with quadratic variation $\langle \beta, \beta \rangle_{u}  = u$, so by L\'{e}vy's theorem it is a Brownian motion. The equality (\ref{logofmart}) can now be rewritten:
\[ \log{\frac{\rho_{t \wedge \tau_{n}}}{x}} + \frac{1-\delta}{2} \int_{0}^{t \wedge \tau_{n}} \frac{ds}{\rho_{s}^{2}} = \beta_{A_{t \wedge \tau_{n}}} - \frac{1}{2} A_{t \wedge \tau_{n}}. \]
Letting $n \to \infty$, we obtain, for all $t < T_{0}$:
\[ \log{\frac{\rho_{t}}{x}} + \frac{1-\delta}{2} \int_{0}^{t} \frac{ds}{\rho_{s}^{2}} = \beta_{A_{t}} - \frac{1}{2} A_{t}. \]
By the asymptotic properties of Brownian motion (see Corollary (1.12), Chapter II in \cite{revuz2013continuous}), we know that, a.s.:
\[ \underset{s \to + \infty}{\limsup} \ \frac{\beta_{s}}{h(s)} = 1 \]
where $h(s) := \sqrt{2 s \log\log s }$. In particular, a.s., there exists $T>0$ such that, for all $t \geq T$, we have $ \beta_{t} \leq 2 h(t) $. Since, a.s., $A_{t} \underset{ t \to T_{0}}{\longrightarrow} + \infty$, we deduce that:
\begin{align*}
\underset{t \to + \infty}{\limsup} \left( \beta_{A_{t}} - \frac{1}{2} A_{t} \right) &\leq  \underset{t \to + \infty}{\limsup} \left( 2 h(A_{t}) - \frac{1}{2} A_{t} \right) \\
& = - \infty .
\end{align*}
Hence, a.s. : 
\[ \log  \left[ \frac{\rho_{t}}{x}\exp \left( \frac{1-\delta}{2} \int_{0}^{t} \frac{ds}{\rho_{s}(x)^{2}} \right) \right] \underset{ t \uparrow T_{0}(x)}{\longrightarrow} - \infty \]
i.e.
\[ \rho_{t} \exp \left( \frac{1-\delta}{2} \int_{0}^{t} \frac{ds}{\rho_{s}(x)^{2}} \right) \underset{ t \uparrow T_{0}(x)}{\longrightarrow} 0 \]
as claimed.
\end{proof}


\subsection{Martingale property of $(D_{t})_{t \geq 0}$}
\label{eta}

Let $\delta \geq 0$ and $x > 0$ be fixed. We show in this section that $(D_{t})_{t \geq 0}$ is an $(\mathcal{F}_{t})_{t \geq 0}$ martingale which, up to a positive constant, corresponds to a Girsanov-type change of probability measure. 

Recall that, by definition:
\begin{equation}
\label{expeta}
D_{t} = \mathbf{1}_{t < T_{0}(x)}  \rho_{t}(x) \exp \left( -\frac {\delta -1}{2} \int_{0}^{t} \frac{ds}{\rho_{s}^{2}} \right). 
\end{equation}

\begin{nota}
For all $a \geq 0$ and $t \geq 0$, we denote by ${P^{a}_{x}} \big\rvert_{ \mathcal{F}_{t}}$ the image of the probability measure ${P^{a}_{x}}$ under the restriction map:
\begin{align*}
(C(\mathbb{R}_{+}), \mathcal{B} ( C(\mathbb{R}_{+})) &\to  (C([0,t]),\mathcal{F}_{t}) \\
w &\mapsto  w \rvert_{[0,t]}
\end{align*}
\end{nota}

\begin{prop}
\label{abscont}
Let $\delta \geq 0$ and $x>0$. Then, for all $t \geq 0$, the law $ {P^{\delta+2}_{x}}_{| \mathcal{F}_{t}}$ is absolutely continuous w.r.t. the law ${P^{\delta}_{x}}_{| \mathcal{F}_{t}}$, and the corresponding Radon-Nikodym derivative is given by:
\[ \frac{dP^{\delta+2}_{x}}{dP^{\delta}_{x}}\biggr\rvert_{\mathcal{F}_{t}} (\rho) \overset{a.s.}{=}  \mathbf{1}_{t < T_{0}(x)} \frac{\rho_{t}(x)}{x} \exp \left( -\frac {\delta -1}{2} \int_{0}^{t} \frac{ds}{\rho_{s}^{2}} \right). \]
\end{prop}

\begin{proof}
 Fix $\epsilon > 0$. Under ${P^{\delta}_{x}}_{| \mathcal{F}_{t}}$, the canonical process $\rho$ stopped at $T_{\epsilon}$ satisfies the following SDE on $[0,t]$:
\[ \rho_{s \wedge T_{\epsilon}} = x + \frac{\delta-1}{2} \int_{0}^{s \wedge T_{\epsilon}} \frac{ds}{\rho_{s}} + B_{s \wedge T_{\epsilon}}. \]
Consider the process $M^{\epsilon}$ defined on $[0,t]$ by:
\[ M^{\epsilon}_{s} := \int_{0}^{s \wedge T_{\epsilon}} \frac{dB_{u}}{\rho_{u}} \] 
$M^{\epsilon}$ is an $L^{2}$ martingale on $[0,t]$. The exponential local martingale thereto associated is:
\[ \mathcal{E}(M^{\epsilon})_{s} = \exp \left(\int_{0}^{s \wedge T_{\epsilon}} \frac{dB_{u}}{\rho_{u}} - \frac{1}{2} \int_{0}^{s \wedge T_{\epsilon}} \frac{du}{\rho_{u}^{2}} \right). \]
Since, by It\^{o}'s lemma:
\[ \log \left( \frac{\rho_{s \wedge T_{\epsilon}}}{x} \right) = \int_{0}^{s \wedge T_{\epsilon}}  \frac{dB_{u}}{\rho_{u}} + \left( \frac{\delta}{2} - 1 \right) \int_{0}^{s \wedge T_{\epsilon}} \frac{du}{\rho_{u}^{2}}, \]
we have:
\begin{align*}
\mathcal{E}(M^{\epsilon})_{s} &= \exp \left[ \log \left(  \frac{\rho_{s \wedge T_{\epsilon}}}{x} \right)  - \frac{\delta-1}{2} \int_{0}^{s \wedge T_{\epsilon}} \frac{du}{\rho_{u}^{2}} \right] \\
&=   \frac{\rho_{s \wedge T_{\epsilon}}}{x} \exp \left[ - \frac{\delta-1}{2} \int_{0}^{s \wedge T_{\epsilon}} \frac{du}{\rho_{u}^{2}} \right].
\end{align*}
Note that 
\begin{align*}
\mathbb{E} \left[ \exp \left( \frac{1}{2} \langle M^{\epsilon}, M^{\epsilon} \rangle_{t} \right) \right] &\leq \exp \left( \frac{t}{2\epsilon} \right) \\
&< \infty
\end{align*}
so that, by Novikov's criterion, $\mathcal{E}(M^{\epsilon})$ is a uniformly integrable martingale on $[0,t]$. So we may consider the probability measure $\mathcal{E}(M^{\epsilon}) P^{\delta}_{x} \big\rvert_{\mathcal{F}_{t}}$. \\
Note also that:
\[ \langle M^{\epsilon}, B \rangle_{t} = \int_{0}^{s \wedge T_{\epsilon}} \frac{du}{\rho_{u}}. \]
Hence, by Girsanov's theorem, under the probability measure $\mathcal{E}(M^{\epsilon}) P^{\delta}_{x} \big\rvert_{\mathcal{F}_{t}}$, the process :
\[ \rho_{s \wedge T_{\epsilon}} - x - \frac{\delta+1}{2} \int_{0}^{s \wedge T_{\epsilon}} \frac{du}{\rho_{u}} \]
is a local martingale, with quadratic variation given by $s \wedge T_{\epsilon}$. Therefore, by Theorem (1.7) in Chapter V of \cite{revuz2013continuous}, there exists, on some enlarged probability space, a Brownian motion $\beta$ such that, a.s.:
\[ \forall s \in [0,t], \ \rho_{s \wedge T_{\epsilon}} = x + \frac{\delta+1}{2} \int_{0}^{s \wedge T_{\epsilon}} \frac{du}{\rho_{u}} + \beta_{s \wedge T_{\epsilon}}. \]
Denote by $\bar{\rho}$ the unique strong solution on $[0,t]$ of the SDE:
\[  \bar{\rho}_{s} = x + \frac{\delta+1}{2} \int_{0}^{s} \frac{du}{\bar{\rho}_{u}} + \beta_{s}. \]
Then, by strong uniqueness of the solution to this SDE, we deduce that, under $\mathcal{E}(M^{\epsilon}) P^{\delta}_{x} \big\rvert_{\mathcal{F}_{t}}$, a.s.:
\[ \forall s\in [0, t], \ s < T_{\epsilon} \implies \rho_{s} = \bar{\rho}_{s} . \]
Since $\bar{\rho}$ has the law of a $\delta +2$-dimensional Bessel process started at $x$, we deduce that, for all $F : C([0,T], \mathbb{R}_{+}) \to \mathbb{R}_{+}$ Borel, we have:
\[ E^{\delta}_{x} \left[ \mathcal{E}(M^{\epsilon}) F(\rho) \mathbf{1}_{t < T_{\epsilon}} \right] = E^{\delta + 2}_{x}[F(\rho) \mathbf{1}_{t < T_{\epsilon}}] \]
i.e.:
\[ E^{\delta}_{x} \left[ \frac{\rho_{t}}{x} \exp \left( - \frac{\delta-1}{2} \int_{0}^{t} \frac{ds}{\rho_{s}^{2}} \right) F(\rho) \mathbf{1}_{t < T_{\epsilon}} \right] = E^{\delta + 2}_{x}[F(\rho) \mathbf{1}_{t < T_{\epsilon}}] . \]
Letting $\epsilon \to 0$, by the monotone convergence theorem, we obtain:
\[  E^{\delta}_{x} \left[ \frac{\rho_{t}}{x} \exp \left( - \frac{\delta-1}{2} \int_{0}^{t} \frac{ds}{\rho_{s}^{2}} \right) F(\rho) \mathbf{1}_{t < T_{0}} \right] = E^{\delta + 2}_{x}[F(\rho) \mathbf{1}_{t < T_{0}}] . \] 
But, since $P^{\delta + 2}_{x}[T_{0} < + \infty] = 0$, this yields:
\[ E^{\delta}_{x} \left[ \frac{\rho_{t}}{x} \exp \left( - \frac{\delta-1}{2} \int_{0}^{t} \frac{ds}{\rho_{s}^{2}} \right) F(\rho) \mathbf{1}_{t < T_{0}} \right] = E^{\delta+2}_{x}[F(\rho)] \]
as stated.
\end{proof}

\begin{rk}
Proposition \ref{abscont} is actually a particular case of a more general result. Indeed, for all $ x > 0$, $ t \geq 0$, and $\delta' \geq \delta \geq 0$, such that $\delta' \geq 2$, $ {P^{\delta'}_{x}}_{| \mathcal{F}_{t}}$ is absolutely continuous w.r.t. the law ${P^{\delta}_{x}}_{| \mathcal{F}_{t}}$, and the corresponding Radon-Nikodym derivative is given by:
\begin{equation}
\label{generalabscont}
 \frac{dP^{\delta'}_{x}}{dP^{\delta}_{x}}\biggr\rvert_{\mathcal{F}_{t}} (\rho) \overset{a.s.}{=}  \mathbf{1}_{t < T_{0}(x)} \left( \frac{\rho_{t}(x)}{x} \right)^{\frac{\delta'-\delta}{2}} \exp \left[ -\frac {\delta' -\delta}{2} \left(\frac{\delta'+\delta}{4} -1 \right) \int_{0}^{t} \frac{ds}{\rho_{s}^{2}} \right]. 
\end{equation}
The proof of this fact is in all respect similar to that of Proposition \ref{abscont} above.
\end{rk}

\begin{cor}
$(D_{t})_{t \geq 0}$ is an $(\mathcal{F}_{t})_{t \geq 0}$ continuous martingale
\end{cor}

\begin{proof}
The process $(D_{t})_{t \geq 0}$  is continuous. Moreover, for all $t \geq 0$, $\frac{1}{x} D_{t}$ is the Radon-Nikodym derivative of ${P^{\delta+2}_{x}}_{| \mathcal{F}_{t}}$ w.r.t. ${P^{\delta}_{x}}_{| \mathcal{F}_{t}}$. Therefore $(\frac{1}{x} D_{t})_{t \geq 0}$ is an $(\mathcal{F}_{t})_{t \geq 0}$ martingale, so $(D_{t})_{t \geq 0}$ is  a martingale as well, and the claim follows.
\end{proof}
 
 
\subsection{Moment estimates for the martingale $(D_{t})_{t \geq 0}$ }

In this section, we prove that the martingale $(D_{t})_{t \geq 0}$ is actually in $L^{p}$ for some $p \geq 1$. We first recall the following fact:

\begin{lm}
\label{mmtsbes}
For all $a \geq 0$, $t \geq 0$, and $m \geq 0$, we have:
\[ E^{a}_{x} (\rho_{t}^{m}) < \infty . \]
\end{lm}

\begin{proof}
Denote by $d$ any integer such that $d \geq a$. By Lemma \ref{weakcomp}, we have:  
\[ E^{a}_{x} (\rho_{t}^{m}) \leq E^{d}_{x} (\rho_{t}^{m}) \]
Since $P^{d}_{x}$ is the law of $(||B_{s}||)_{s \geq 0}$, where $(B_{s})_{s \geq 0}$ is a $d$-dimensional Brownian motion and $|| \cdot ||$ is the Euclidean norm in $\mathbb{R}^{d}$ (see \cite{revuz2013continuous}, Chapter 11), this inequality can be rewritten as:
\[ E^{a}_{x} (\rho_{t}^{m}) \leq \mathbb{E} \left( ||B_{t}||^{m} \right) \]
Since $B_{t}$ is a Gaussian random variable, $\mathbb{E} \left( ||B_{t}||^{m} \right)$ is finite, and the result follows.
\end{proof}

\begin{prop}
\label{lp}
$(D_{t})_{t \geq 0}$ is an $L^{p}$ martingale for all finite positive number $p$ such that $p \leq p(\delta)$, where $p(\delta) \in [1, + \infty]$ is given by:
\begin{align}
\label{pdelta} 
p(\delta) :=  \begin{cases} 
                     \frac{(2-\delta)^{2}}{4(1-\delta)} \ &\text{if} \ \delta < 1, \\
                     + \infty \ &\text{if} \ \delta \geq 1. \\
                     \end{cases} 
\end{align}
Moreover the above statement is sharp: for $\delta < 1$ and $t >0$, the random variable $D_{t}$ is not in $L^{p}$ for $p > p(\delta)$. 
\end{prop}

\begin{rk}
We emphasize that $p$ is \textit{finite} in the above result. Indeed $D_{t}$ is never in $L^{\infty}$ even if $\delta \geq 1$; for example, when $\delta = 1$, $D_{t} = \rho_{t}  \mathbf{1}_{t < T_{0}(x)}$ which is clearly not bounded a.s. .
\end{rk} 

\begin{proof}[Proof of Prop \ref{lp}]
If $\delta \geq 1$, then, for all $t \geq 0$, $D_{t} \leq \rho_{t}$. Hence, for all $p \in (0, + \infty)$:
\[ \mathbb{E} \left( D_{t}^{p} \right) \leq E^{\delta}_{x} ({\rho_{t}}^{p}) \]
which is finite by Lemma \ref{mmtsbes}. 

On the other hand, if $\delta \in [0,1)$, then, for all $t >0$ and $p>0$, we have:
\[ \mathbb{E} \left( D_{t}^{p} \right) = E^{\delta}_{x} \left[ \mathbf{1}_{t < T_{0}} \ {\rho_{t}}^{p} \ \exp \left( - p \frac{\delta-1}{2} \int_{0}^{t} \frac{ds}{\rho_{s}^{2}} \right) \right]. \]
By the absolute continuity relation (\ref{generalabscont}) applied with $\delta':=2$, the latter equals:
\[
E^{2}_{x} \left[ x^{\frac{2-\delta}{2}} \ {\rho_{t}}^{p + \frac{\delta-2}{2} } \ \exp \left( \underbrace{\left( - p \frac{\delta-1}{2} - \frac{(\delta-2)^{2}}{8} \right)}_{\text{$:= A(p)$} }  \int_{0}^{t} \frac{ds}{\rho_{s}^{2}}\right) \right]. \]
For $p = p(\delta)$, $A(p) = 0$, so that:
\begin{align*}
\mathbb{E} \left[ D_{t}^{p(\delta)} \right] &= E^{2}_{x} \left[ x^{\frac{2-\delta}{2}} {\rho_{t}}^{p(\delta) +\frac{\delta-2}{2}} \right] \\
&= x^{1 - \frac{\delta}{2}} \ E^{2}_{x} \left[ {\rho_{t}}^{p(\delta) + \frac{\delta}{2} - 1} \right]. 
\end{align*}
Since $\frac{\delta}{2} + p(\delta) - 1 \geq 0$, by Lemma \ref{mmtsbes}, the last quantity is finite. Hence $D_{t}$ is indeed in $L^{p(\delta)}$. 

Suppose now that $p = p(\delta) + r$ for some $r >0$. We show that $D_{t} \notin L^{p}$. We have:
\begin{align*}
\mathbb{E} \left[ D_{t}^{p} \right] &= E^{2}_{x} \left[ x^{\frac{2-\delta}{2}} \ {\rho_{t}}^{p + \frac{\delta-2}{2} } \ \exp \left( \left( - p \frac{\delta-1}{2} - \frac{(\delta-2)^{2}}{8} \right)  \int_{0}^{t} \frac{ds}{{\rho_{s}}^{2}}\right) \right] \\
&= x^{1-\frac{\delta}{2}} E^{2}_{x} \left[ {\rho_{t}}^{p + \frac{\delta}{2} -1} \ \exp \left( \frac{1-\delta}{2} r \int_{0}^{t} \frac{ds}{{\rho_{s}}^{2}} \right) \right] 
\end{align*}
We claim that the last quantity is infinite. Indeed, first note that by Jensen's inequality and Fubini, for any $C> 0$ we have:
\begin{align*} 
E^{2}_{x} \left[ \exp \left( C \int_{0}^{t}  \frac{ds}{{\rho_{s}}^{2}} \right) \right] &\geq  \exp \left( C \int_{0}^{t}  E^{2}_{x} \left(\rho_{s}^{-2} \right) ds \right) 
\end{align*}
and the right-hand side is infinite since, for all $s > 0$, $E^{2}_{x} \left(\rho_{s}^{-2} \right) = + \infty$ (indeed, by formula (\ref{sgpdensity}), the transition density $p^{2}_{s}(x,y)$ does not integrate $y^{-2}$ as $y \to 0$). Therefore: 
\begin{equation} 
\label{diverging}
E^{2}_{x} \left[ \exp \left( C \int_{0}^{t}  \frac{ds}{{\rho_{s}}^{2}} \right) \right]  = + \infty 
\end{equation}
Consider now any $c >0$ and $a, b >0$ such that $\frac{1}{a} + \frac{1}{b} = 1$.  By (\ref{diverging}) and H\"{o}lder's inequality, we have:
\begin{align*} 
 + \infty &= E^{2}_{x} \left[ \exp \left( \frac{1-\delta}{2a} r \int_{0}^{t}  \frac{ds}{{\rho_{s}}^{2}} \right) \right] \\
& \leq E^{2}_{x} \left[ {\rho_{t}}^{ac} \exp \left( \frac{1-\delta}{2} r \int_{0}^{t}  \frac{ds}{{\rho_{s}}^{2}} \right) \right]^{1/a} E^{2}_{x} \left[ \rho_{t}^{-bc} \right]^{1/b}   
\end{align*}
Set $c =   \frac{\frac{\delta}{2} + p - 1}{\frac{\delta}{2} + p}$, $a = \frac{\delta}{2} + p$,  and $b = \frac{\frac{\delta}{2} + p}{\frac{\delta}{2} + p-1}$ . Remark that $\frac{\delta}{2} + p - 1>0$ since $p > p(\delta) \geq 1$, so that this choice for $c$, $a$, and $b$ makes sense. We obtain:
\[ E^{2}_{x} \left[ {\rho_{t}}^{\frac{\delta}{2} + p - 1 } \ \exp \left( \frac{1-\delta}{2} r \int_{0}^{t} \frac{ds}{{\rho_{s}}^{2}} \right) \right] ^{\frac{1}{\frac{\delta}{2} + p}} E^{2}_{x} \left[ \rho_{t}^{-1} \right]^{\frac{\frac{\delta}{2} + p - 1}{\frac{\delta}{2} + p}} = +\infty \]
By the comparison lemma \ref{weakcomp} and the expression (\ref{sgpdensity}) for the transition density of the Bessel process, we have 
\begin{align*} 
E^{2}_{x} \left[ \rho_{t}^{-1} \right] \leq E^{2}_{0} \left[ \rho_{t}^{-1} \right]  & = \int_{0}^{\infty} \frac{1}{t} \exp \left( - \frac{y^{2}}{2t} \right) dy 
\end{align*}
so that $E^{2}_{x} \left[ \rho_{t}^{-1} \right] < + \infty$. 
Therefore, we deduce that  :
\[ E^{2}_{x} \left[ {\rho_{t}}^{\frac{\delta}{2} + p - 1  } \ \exp \left( \frac{1-\delta}{2} r \int_{0}^{t} \frac{ds}{{\rho_{s}}^{2}} \right) \right]  = + \infty \]
as claimed. Hence $D_{t} \notin L^{p}$ for $p > p(\delta)$.
\end{proof}


\section{A Bismut-Elworthy-Li formula for the Bessel processes}

We are now in position to provide a probabilistic interpretation of the right-hand-side of equation (\ref{bel2}) in Theorem \ref{thmana}. 

Let $\delta > 0$, and $x >0$. As we saw in the previous section, the process $(\eta_{t}(x))_{t \geq 0}$ may blow up at time $T_{0}$, so that the stochastic integral  $\int_{0}^{t} \eta_{s} (x) dB_{s}$ is a priori ill-defined, at least for $\delta \in (0, 1)$. However, it turns out that we can define the latter process rigorously as a local martingale.

\begin{prop}
\label{stochint}
Suppose that $\delta >0$. Then the stochastic integral process $\int_{0}^{t} \eta_{s} dB_{s}$ is well-defined as a local martingale and is indistinguishable from the continuous martingale $D_{t} - x$. 
\end{prop}

\begin{proof}
We first treat the case $\delta \geq 2$, which is much easier to handle. In that case, $\eta_{t} \in [0,1]$ for all $t \geq 0$,  so that the stochastic integral $\int_{0}^{t} \eta_{s} dB_{s}$ is clearly well-defined as an $L^{2}$ martingale. Moreover, since $T_{0} = + \infty$ a.s.,  by It\^{o}'s lemma we have:
\begin{align*}
D_{t} = \rho_{t} \eta_{t} &= x +  \int_{0}^{t} \eta_{s} \ d \rho_{s} +  \int_{0}^{t} \rho_{s} \ d \eta_{s} \\
&= x + \int_{0}^{t} \eta_{s} \left( \frac{\delta-1}{2} \frac{ds}{\rho_{s}} + dB_{s} \right) -  \int_{0}^{t} \rho_{s} \frac{\delta-1}{2} \frac{\eta_{s}}{{\rho_{s}}^{2}} ds \\
&= x + \int_{0}^{t} \eta_{s} dB_{s}
\end{align*} 
so the claim follows. 

Now suppose that $\delta \in (0,2)$ and fix an $\epsilon > 0$. Recall that $T_{\epsilon}(x) := \inf \{ t \geq 0,  \rho_{t}(x) \leq \epsilon \}$ and note that, since $T_{\epsilon} < T_{0}$, the stopped process $\eta^{T_{\epsilon}}$ is continuous on $\mathbb{R}_{+}$, so that the stochastic integral $\int_{0}^{t \wedge T_{\epsilon}(x)} \eta_{s}(x) dB_{s}$ is well-defined as a local martingale. Using as above It\^{o}'s lemma, but this time with the stopped processes $\rho^{T_{\epsilon}}$ and $\eta^{T_{\epsilon}}$, we have:
\begin{equation}
\label{approximate_eq}
 \int_{0}^{t \wedge T_{\epsilon}} \eta_{s} dB_{s} =  D_{t \wedge T_{\epsilon}} - x .
 \end{equation}
Our aim would be to pass to the limit $\epsilon \to 0$ in this equality. By continuity of $D$, as $\epsilon \to 0$, $D_{t \wedge T_{\epsilon}}$ converges to $D_{t \wedge T_{0}}$ = $D_{t}$ almost-surely. So the right-hand side of (\ref{approximate_eq}) converges to $D_{t} - x$ almost-surely. 

The convergence of the left-hand side to a stochastic integral is more involved, since we first have to prove that the stochastic integral $\int_{0}^{t} \eta_{s} dB_{s}$ is indeed well-defined as a local martingale. For this, it suffices to prove that, almost-surely:
\[ \forall t \geq 0, \qquad \int_{0}^{t} \eta_{s}^{2} \ ds < \infty .\]
We actually prove the following stronger fact. For all $t \geq 0$ 
\begin{equation}
\label{bracket}
 \mathbb{E} \left[ \left( \int_{0}^{t} \eta_{s}^{2} ds \right)^{p/2} \right] < \infty 
\end{equation}
for all finite positive number $p$ such that $p \in (1,p(\delta)]$.   Indeed, applying successively the Burkholder-Davis-Gundy (BDG) inequality and Doob's inequality to the martingale $\int_{0}^{T_{\epsilon} \wedge \cdot} \eta_{s} dB_{s}$ , we have:
\begin{align*}
 \mathbb{E} \left[ \left( \int_{0}^{t \wedge T_{\epsilon}} \eta_{s}^{2} ds \right)^{p/2} \right] & \leq C_{p}  \ \mathbb{E} \left[ \underset{s \leq t \wedge T_{\epsilon}}{\sup} \left| \int_{0}^{s} \eta_{u} dB_{u} \right|^{p} \right] \\
& = C_{p} \ \mathbb{E} \left[ \underset{s \leq t \wedge T_{\epsilon}}{\sup} |D_{s}-x|^{p} \right] \\
& \leq C_{p} \left( \frac{p}{p-1} \right)^{p}  \mathbb{E} \left[ | D_{t \wedge T_{\epsilon}}-x |^{p} \right] \\
\end{align*}
where $C_{p}$ is a constant depending only on $p$. Now, since $(D_{t} - x)_{t \geq 0}$ is a  continuous martingale, by the optional stopping theorem and Jensen's inequality, we have:
\[ \mathbb{E} \left[ | D_{t \wedge T_{\epsilon}}-x |^{p} \right] \leq \mathbb{E} (|D_{t}-x| ^{p}) \] 
and the right-hand side is finite because $D_{t}$ is in $L^{p}$ . Hence, letting $\epsilon \to 0$ in the above, by the monotone convergence theorem we deduce that :
 \[  \mathbb{E} \left[ \left( \int_{0}^{t \wedge T_{0}} \eta_{s}^{2} ds \right)^{p/2} \right] < \infty \]
 But since $\eta_{t} = 0$ for all $t \geq T_{0}$, this implies the bound (\ref{bracket}), and hence the stochastic integral  $\int_{0}^{t} \eta_{s} dB_{s}$ is well-defined as a local martingale. Moreover, for all $t \geq 0$, by the BDG inequality, we have:
 \begin{align*}
\mathbb{E} \left[ \left( \int_{0}^{t} \eta_{s} dB_{s} - \int_{0}^{t \wedge T_{\epsilon}} \eta_{s} dB_{s} \right) ^{p} \right] &=  \mathbb{E} \left[ \left( \int_{t \wedge T_{\epsilon} }^{t \wedge T_{0}} \eta_{s} dB_{s}\right) ^{p} \right] \\
 & \leq c_{p} \ \mathbb{E} \left[ \left( \int_{t \wedge T_{\epsilon} }^{t \wedge T_{0}} \eta_{s}^{2} ds \right) ^{p/2} \right]
 \end{align*} 
where $c_{p}$ is some constant depending only on $p$. Now, by the dominated convergence theorem, the last quantity above goes to $0$ as $\epsilon \to 0$ , and hence:
\[  \int_{0}^{t \wedge T_{\epsilon}} \eta_{s} dB_{s} \underset{ \epsilon \to 0}{\longrightarrow}   \int_{0}^{t} \eta_{s} dB_{s} \]
in $L^{p}$. Hence, the left-hand side of equality (\ref{approximate_eq}) converges in $L^{p}$ to the stochastic integral $\int_{0}^{t} \eta_{s} dB_{s}$. Letting $\epsilon \to 0$ in that equality, we thus obtain :
\[ \int_{0}^{t} \eta_{s} dB_{s} = D_{t} - x \]
as claimed.
\end{proof}

Using the above proposition, Theorem \ref{thmana} can now be interpreted probabilistically as a Bismut-Elworthy-Li formula. 

\begin{thm} [Bismut-Elworthy-Li formula]
\label{thmprobinterp}
Let $\delta > 0$. Then, for all $T>0$, and all $F : \mathbb{R}_{+} \to \mathbb{R}$ bounded and Borel, the function $x \to P^{\delta}_{t} F (x)$ is differentiable on $\mathbb{R}_{+}$, and for all $x > 0$:
\begin{equation}
\label{bel3}
\frac{d}{dx}  P^{\delta}_{T} F (x) = \frac{1}{T} \ \mathbb{E} \left[ F(\rho_{t}(x)) \left( \int_{0}^{T} \eta_{s}(x) dB_{s} \right) \right].
\end{equation}
\end{thm}

\begin{proof}
By Theorem \ref{thmana}, the differentiablity property holds, and we have:
\[ \frac{d}{dx}  P^{\delta}_{T} F (x) = \frac{x}{T} \left[ P^{\delta +2}_{T} F (x) - P^{\delta}_{T} F (x) \right]. \]
Moreover, by Proposition \ref{abscont}, for all $x > 0$:
\begin{align*}
P^{\delta +2}_{T} F (x) - P^{\delta}_{T} F (x) = E^{\delta}_{x} \left[ F(\rho_{T}) \left( \frac{D_{T}}{x} - 1 \right) \right]   
\end{align*}
and, by Proposition \ref{stochint}, we have:
\[ E^{\delta}_{x} \left[ F(\rho_{T}) \left( \frac{D_{T}}{x} - 1 \right) \right]   
 = \frac{1}{x}  \mathbb{E} \left[ F(\rho_{T}(x))  \left( \int_{0}^{T} \eta_{s}(x) dB_{s}  \right) \right] \]
so equality (\ref{bel3}) follows. 
\end{proof}

Using the Bismut-Elworthy-Li formula, we are now able to sharpen the Strong Feller estimate obtained in equation (\ref{feller}) above.  

\begin{cor}
Let $T > 0$ and $\delta \geq 2(\sqrt{2} -1)$. Then, for all $R > 0$, there exists a constant $C>0$ such that, for all $x, y \in [0,R]$ and  $F: \mathbb{R}_{+} \to \mathbb{R}$ bounded and Borel, we have:
\begin{equation}
\label{betterbound}
 | P^{\delta}_{T} F (x) - P^{\delta}_{T} F (y)| \leq \frac{ C ||F||_{\infty}}{T^{\alpha(\delta)}} \ |y-x| 
\end{equation}
where the exponent $\alpha(\delta) \in [\frac{1}{2},1)$ is given by:
\begin{align*}
\label{alphadelta} 
\alpha(\delta) :=  \begin{cases} 
                     \frac{1}{2} + \frac{1-\delta}{2-\delta} \ &\text{if} \ \delta \in [2(\sqrt{2}-1), 1], \\
                     1/2 \ &\text{if} \ \delta \geq 1. \\
                     \end{cases} 
\end{align*}
\end{cor}

\begin{proof}
Let $x > 0$. By Theorem \ref{thmprobinterp}, we have :
\[ \frac{d}{dx}  P^{\delta}_{T} F (x) = \frac{1}{T} \ \mathbb{E} \left[ F(\rho_{t}(x)) \left( \int_{0}^{T} \eta_{s}(x) dB_{s} \right) \right]. \]
so that:
\[ \left| \frac{d}{dx}  P^{\delta}_{T} F (x) \right| \leq \frac{||F||_{\infty}}{T} \ \mathbb{E} \left[\left| \int_{0}^{T} \eta_{s}(x) dB_{s} \right|\right]. \]
We now bound the quantity $ \mathbb{E} \left[\left| \int_{0}^{T} \eta_{s}(x) dB_{s} \right|\right]$. If $\delta \geq 1$, then the process $(\eta_{s}(x))_{s \geq 0}$ takes values in $[0,1]$, so that, using the Cauchy-Schwarz inequality and It\^{o}'s isometry formula, we have :
\[ \mathbb{E} \left[  \left| \int_{0}^{T} \eta_{s}(x) dB_{s} \right| \right] \leq \sqrt{ \mathbb{E} \left( \int_{0}^{T} \eta_{s}(x)^{2} ds \right)} \leq \sqrt{T}. \]
Therefore:
\[ \left| \frac{d}{dx}  P^{\delta}_{T} F (x) \right| \leq \frac{||F||_{\infty}}{\sqrt{T}} \] 
and the claim follows with $C =1$. 

Suppose now that $\delta \in [2(\sqrt{2}-1), 1)$, and let $p := p(\delta)$ as in (\ref{pdelta}).  Note that $p \geq 2$ by the assumption on $\delta$. By Jensen's inequality: 
\[ \mathbb{E} \left[  \left| \int_{0}^{T} \eta_{s}(x) dB_{s} \right| \right] \leq \left( \mathbb{E} \left| \int_{0}^{T} \eta_{s}(x) dB_{s} \right|^{p} \right)^{1/p} \]
Now, applying successively the BDG inequality, Jensen's inequality and the absolute continuity  relation (\ref{generalabscont}) between $P^{2}_{x}$  and $P^{\delta}_{x}$, we have, for some constant $c_{p}$ depending only on $p$:
\begin{align*}
\mathbb{E} \left[  \left| \int_{0}^{T} \eta_{s}(x) dB_{s} \right|^{p} \right] & \leq c_{p} \ \mathbb{E} \left[\left( \int_{0}^{T} \eta_{s}(x)^{2} ds \right)^{p/2}\right] \\
& \leq c_{p} \ T^{p/2-1} \ \mathbb{E} \left( \int_{0}^{T} \eta_{s}(x)^{p} \ ds \right) \\ 
& \leq c_{p} \ T^{p/2-1} \int_{0}^{T} E^{\delta}_{x} (\eta_{s}^{p}) \ ds \\
& = c_{p} \ T^{p/2-1} \int_{0}^{T} E^{2}_{x} \left[  \left( \frac{\rho_{s}}{x} \right)^\frac{\delta-2}{2} \exp \left( \left(\frac{1- \delta}{2}p - \frac{(2-\delta)^{2}}{8} \right) \int_{0}^{s} \frac{du}{\rho_{u}^{2}}  \right) \right]ds \\
& = c_{p} \ T^{p/2-1} \int_{0}^{T} E^{2}_{x} \left[ \left( \frac{\rho_{s}}{x} \right)^\frac{\delta-2}{2}  \right]ds
\end{align*}
where the last equality follows from the fact that $\frac{1- \delta}{2}p - \frac{(2-\delta)^{2}}{8} = 0$ for $p=p(\delta)$. Now, since $\frac{\delta-2}{2} \leq 0$, by the comparison lemma \ref{weakcomp}, as well as the scaling property of the Bessel processes (see, e.g., Remark 3.7 in \cite{zambotti2017random}), for all $s \in [0,T]$, we have:
\[ E^{2}_{x} \left[ \rho_{s}^{\frac{\delta-2}{2} } \right] \leq  E^{2}_{0} \left[ \rho_{s}^\frac{\delta-2}{2} \right] = s^{\frac{\delta-2}{4}} E^{2}_{0} \left[ \rho_{1}^\frac{\delta-2}{2} \right] .\]
Let $c := E^{2}_{0} \left[ \rho_{1}^\frac{\delta-2}{2} \right]$. Using formula (\ref{sgpdensity}), we have:
\[ c = \int_{0}^{\infty} y^{\delta/2} \exp \left( -\frac{y^{2}}{2} \right) \ dy < \infty . \]
Hence:
\begin{align*}
 \int_{0}^{T} E^{2}_{x} \left[ \left( \frac{\rho_{s}}{x} \right)^{\frac{\delta}{2} -1}  \right]ds &\leq c \ x^{1- \frac{\delta}{2}}  \int_{0}^{T} s^{\frac{\delta-2}{4}} ds \\
&\leq \frac{4 \ c}{\delta+2} \ x^{1- \frac{\delta}{2}} \ T^{\frac{\delta+2}{4}}.
\end{align*}
Therefore, we obtain:
\begin{align*}
\mathbb{E} \left[  \left| \int_{0}^{T} \eta_{s}(x) dB_{s} \right|^{p} \right] & \leq K \ x^{1- \frac{\delta}{2}} \ T^{\frac{p}{2} -1} \ T^{\frac{\delta+2}{4}} \\
& \leq  K \ x^{1- \frac{\delta}{2}} \ T^{\frac{p}{2} +\frac{\delta-2}{4}} \\
\end{align*}
where $K$ is a constant depending only on $\delta$. Hence
\[ \mathbb{E} \left[  \left| \int_{0}^{T} \eta_{s}(x) dB_{s} \right| \right] \leq K^{1/p} \ x^{ \frac{1}{p}(1- \frac{\delta}{2})} \ T^{\frac{1}{2} + \frac{\delta-2}{4p}} .\]
Note that, since $p=p(\delta)$, we have $\frac{1}{p}(1- \frac{\delta}{2}) = \frac{2(1-\delta)}{2 - \delta}$, and $\frac{\delta-2}{4p} = - \frac{1-\delta}{2-\delta}$. Therefore, we obtain: 
\[ \left| \frac{d}{dx}  P^{\delta}_{T} F (x) \right| \leq K^{1/p} \ x^{ 2\frac{1-\delta}{2 - \delta}} \ ||F||_{\infty}  \ T^{-\frac{1}{2} - \frac{1-\delta}{2-\delta}} .\]
Therefore, given $R>0$, one has for all $x \in [0,R]$:
\[ \left| \frac{d}{dx}  P^{\delta}_{T} F (x) \right| \leq C \frac{||F||_{\infty}}{T^{\alpha(\delta)}} \]
with $C:=  K^{1/p} \ R^{2\frac{1-\delta}{2 - \delta}}$. This yields the claim.

\end{proof}
 
\begin{rk}
\label{open_pbm}
In the above proposition, the value $2(\sqrt{2} -1) $ that appears is the smallest value of $\delta$ for which $\eta$ is in $L^{2}$. For $\delta < 2(\sqrt{2} -1)$, $\eta$ is no longer in $L^{2}$ but only in $L^{p}$ for $p = p(\delta) < 2$, so that we cannot apply Jensen's inequality to bound the quantity $\mathbb{E} \left[  \left( \int_{0}^{T} \eta_{s}(x)^{2} ds \right)^{p/2} \right]$ anymore. It seems reasonable to expect that the bound (\ref{betterbound}) holds also for $\delta < 2(\sqrt{2} -1)$, although we do not have a proof of this fact.
\end{rk}  

\section{Appendix}
In this Appendix, we prove Proposition \ref{derflow}. Recall that we still denote by $(\rho_{t}(x))_{t, x \geq 0}$ the process $(\tilde{\rho}^{\delta}_{t}(x))_{t,x \geq 0}$ constructed in Proposition \ref{modif}.
 
\begin{lm}
\label{bicont}
For all rational numbers $\epsilon, \gamma > 0$, let:
\[ \mathcal{U}_{\gamma}^{\epsilon} := [0, T_{\epsilon}(\gamma)) \times (\gamma, +\infty) \]
and set:
\[ \mathcal{U} := \underset{\epsilon, \gamma \in \mathbb{Q}^{*}_{+}}{ \bigcup}  \mathcal{U}_{\gamma}^{\epsilon}. \]
 Then, a.s., the function $(t,x) \mapsto {\rho}_{t}(x)$ is continuous on the open set $\mathcal{U}$.
\end{lm}

\begin{proof}
By patching, it suffices to prove that, a.s., the function $(t,x) \mapsto {\rho}_{t}(x)$ is continuous on each $\mathcal{U}_{\gamma}^{\epsilon}$, where $\epsilon, \gamma \in \mathbb{Q}^{*}_{+}$. 

Fix  $\epsilon, \gamma \in \mathbb{Q}^{*}_{+}$, and let $x, y \in (\gamma, + \infty) \cap \mathbb{Q}$. We proceed to show that, a.s., for all $t \leq s < T_{\epsilon}(\gamma)$ the following inequality holds:
\begin{equation}
\label{continuity}
 |{\rho}_{t}(x) - {\rho}_{s}(y)| \leq   |x-y| \exp\left( \frac{|\delta-1|}{2\epsilon^{2}} t \right) +  \frac{|\delta-1|}{2\epsilon} |s-t| + |B_{s} - B_{t}| .
\end{equation}
Since $T_{\epsilon}(\gamma) < T_{0}(\gamma)$, a.s., for all $t \leq s \leq T_{\epsilon}(\gamma)$, we have:
\[ \forall  \tau \in [0,t], \qquad {\rho}_{\tau}(x) = x + \frac{\delta - 1}{2} \int_{0}^{\tau} \frac{du}{{\rho}_{u}(x)} + B_{\tau} \]
as well as
\[ \forall  \tau \in [0,s], \qquad {\rho}_{\tau}(y) = y + \frac{\delta - 1}{2} \int_{0}^{\tau} \frac{du}{{\rho}_{u}(y)} + B_{\tau} \]
and hence :
\[ \forall \tau \in [0,t], \qquad |\rho_{\tau}(x) - \rho_{\tau}(y)| \leq |x-y| + \frac{|\delta-1|}{2} \int_{0}^{\tau}\frac{ |\rho_{u}(x) - \rho_{u}(y)|}{\rho_{u}(x) \rho_{u}(y)} du . \]
By the monotonicity property of $\rho$, we have, a.s., for all $t,s$ as above and $u \in [0, s]$:
\begin{equation}
\label{bound_below}
 {\rho}_{u}(x) \wedge {\rho}_{u}(y) \geq {\rho}_{u}(\gamma) \geq \epsilon 
\end{equation}
so that:
\[ \forall \tau \in [0,t], \quad |{\rho}_{\tau}(x) - {\rho}_{\tau}(y)| \leq |x-y| + \frac{|\delta-1|}{2} \int_{0}^{\tau}\frac{ |{\rho}_{u}(x) - {\rho}_{u}(y)|}{\epsilon^{2}} du, \]
which, by Gr\"{o}nwall's inequality, implies that:
\begin{equation}
\label{ineq1}
 |{\rho}_{t}(x) - {\rho}_{t}(y)| \leq |x-y| \exp \left( \frac{|\delta-1|}{2\epsilon^{2}} t \right).
\end{equation}
Moreover, we have:
\[ {\rho}_{s}(y) - {\rho}_{t}(y) = \frac{\delta-1}{2} \int_{t}^{s} \frac{du}{{\rho}_{u}(y)} + B_{s} - B_{t} \]
which, by (\ref{bound_below}), entails the inequality:
\begin{equation}
\label{ineq2}
|{\rho}_{s}(y) - {\rho}_{t}(y)| \leq  \frac{|\delta-1|}{2\epsilon} |s-t| + |B_{s} - B_{t}| .
\end{equation}
Putting inequalities (\ref{ineq1}) and (\ref{ineq2}) together yields the claimed inequality (\ref{continuity}). Hence, we have, a.s., for all rationals $x, y > \gamma$ and all $t \leq s < T_{\epsilon}(\gamma)$:
\[ |{\rho}_{t}(x) - {\rho}_{s}(y)|  \leq   |x-y| \exp \left( \frac{|\delta-1|}{2\epsilon^{2}} t  \right) +  \frac{\delta-1}{2} |s-t| + |B_{s} - B_{t}| \]
and, by density of $\mathbb{Q} \cap (\gamma, + \infty)$ in $(\gamma, + \infty)$, this inequality remains true for all $x, y > \gamma$. Since, a.s., $t \mapsto B_{t}$ is continuous on $\mathbb{R}_{+}$, the continuity of $\rho$ on $\mathcal{U}_{\gamma}^{\epsilon}$ is proved.
\end{proof}

\begin{cor}
\label{assde}
Almost-surely, we have:
\begin{equation}
\label{assdeeq}
 \forall x \geq 0, \quad \forall t \in [0, T_{0}(x)), \qquad {\rho}_{t}(x) = x + \frac{\delta - 1}{2} \int_{0}^{t} \frac{du}{{\rho}_{u}(x)} + B_{t} .
 \end{equation}
\end{cor}

\begin{rk}
We have already remarked in Section \ref{basicsection} that, for all fixed $x \geq 0$, the process $(\rho_{t}(x))_{t \geq 0}$ satisfies the SDE (\ref{sdebessel}). By contrast, the above Corollary shows the stronger fact that, considering the modification $\tilde{\rho}$ of the Bessel flow constructed in Proposition \ref{modif} above, a.s., for \textit{each} $x \geq 0$, the path $(\tilde{\rho}_{t}(x))_{t \geq 0}$  still satisfies relation (\ref{sdebessel}).
\end{rk}

\begin{proof}
Consider an almost-sure event $\mathcal{A} \in \mathcal{F}$ as in Remark \ref{ascont}. On the event $\mathcal{A}$, for all $r \in \mathbb{Q}_{+}$, we have:
\[ \forall t \in [0, T_{0}(r)), \qquad {\rho}_{t}(r) = r + \frac{\delta-1}{2}\int_{0}^{t} \frac{du}{{\rho}_{u}(r)} + B_{t}. \]
Denote by $\mathcal{B} \in \mathcal{F}$ any almost-sure event on which $\rho$ satisfies the monotonicity property (\ref{comp}). We show that, on the event $\mathcal{A} \cap \mathcal{B}$, the property (\ref{assdeeq}) is satisfied. 

Suppose $\mathcal{A} \cap \mathcal{B}$ is fulfilled, and let $x \geq 0$. Then for all $r \in \mathbb{Q}$ such that $r \geq x$, we have:
\[ \forall t \geq 0, \qquad  \rho_{t} (x) \leq \rho_{t}(r) \]
so that $T_{0}(r) \geq T_{0}(x)$. Hence, for all $t \in [0,T_{0}(x))$, we have in particular $t \in [0, T_{0}(r))$, so that:
\[{\rho}_{t}(r) = r + \frac{\delta-1}{2}\int_{0}^{t} \frac{du}{{\rho}_{u}(r)} + B_{t} . \]
Since, for all $u \in [0,t]$, ${\rho}_{u}(r) \downarrow {\rho}_{u}(x)$ as $r \downarrow x$ with $r \in \mathbb{Q}$, by the monotone convergence theorem, we deduce that:
\[ \int_{0}^{t} \frac{du}{{\rho}_{u}(r)} \longrightarrow \int_{0}^{t} \frac{du}{{\rho}_{u}(x)} \]
as $r \downarrow x$ with $r \in \mathbb{Q}$. Hence, letting $r \downarrow x$ with $r \in \mathbb{Q}$ in the above equation, we obtain:
\[ {\rho}_{t}(x) = x + \frac{\delta-1}{2}\int_{0}^{t} \frac{du}{{\rho}_{u}(x)} + B_{t} . \]
This yields the claim.
\end{proof}

One of the main difficulties for proving Proposition \ref{derflow} arises from the behavior of $\rho_{t}(x)$ at $t= T_{0}(x)$. However we will circumvent this problem by working away from the event $t = T_{0}(x)$. To do so, we will make use of the following property.

\begin{lm}
Let $\delta <2$ and $x \geq 0$. Then the function $y \mapsto T_{0}(y)$ is a.s. continuous at $x$.
\end{lm}

\begin{proof}
The function $y \mapsto T_{0}(y)$ is nondecreasing over $\mathbb{R}_{+}$. Hence, if $x > 0$, it has left- and right-sided limits at x,  $T_{0}(x^{-})$ and  $T_{0}(x^{+})$, satisfying:
\begin{equation}
\label{leftrightlmts}
T_{0}(x^{-}) \leq T_{0}(x) \leq T_{0}(x^{+}). 
\end{equation}
Similarly, if $x = 0$, there exists a right-sided limit $T_{0}(0^{+})$ satisfying $T_{0}(0) \leq T_{0}(0^{+})$. Suppose, e.g., that $x > 0$. Then we have:
\begin{equation}
\label{laplacetransf}
\mathbb{E} \left( e^{- T_{0}(x^{+})} \right) \leq \mathbb{E} \left(e^{- T_{0}(x)} \right) \leq \mathbb{E} \left(e^{- T_{0}(x^{-})} \right).
\end{equation}
Now, by the scaling property of the Bessel processes (see, e.g., Remark 3.7 in \cite{zambotti2017random}), for all $y \geq 0$,  the following holds:
\[ (y \rho_{t}(1))_{t \geq 0}  \overset{(d)}{=} (\rho_{y^{2}t}(y))_{t \geq 0}, \]
so that $T_{0}(y) \overset{(d)}{=} y^{2} T_{0}(1)$. Therefore, using the dominated convergence theorem, we have:
\begin{align*}
\mathbb{E} \left( e^{- T_{0}(x^{+})} \right) &= \underset{y \downarrow x}{\lim} \ \mathbb{E} \left( e^{- T_{0}(y)} \right) \\
&= \underset{y \downarrow x}{\lim} \ \mathbb{E} \left( e^{- y^{2} T_{0}(1)} \right) \\
&= \mathbb{E} \left( e^{- x^{2} T_{0}(1)} \right) \\
&= \mathbb{E} \left( e^{- T_{0}(x)} \right).
\end{align*}
Similarly, we have $\mathbb{E} \left( e^{- T_{0}(x^{-})} \right) = \mathbb{E} \left( e^{- T_{0}(x)} \right)$. Hence the inequalities (\ref{laplacetransf}) are actually equalities; recalling the original inequality (\ref{leftrightlmts}), we deduce that $T_{0}(x^{-}) = T_{0}(x) = T_{0}(x^{+})$ a.s.. Similarly, if $x=0$, we have $T_{0}(0) = T_{0}(0^{+})$ a.s. 
\end{proof} 

Before proving Proposition \ref{derflow}, we need a coalescence lemma, which will help us prove that the derivative of $\rho_{t}$ at $x$ is $0$ if $t > T_{0}(x)$:

\begin{lm}
\label{traj}
Let $x, y \geq 0$, and let $\tau$ be a nonnegative $(\mathcal{F}_{t})_{t \geq 0}$-stopping time. Then, almost-surely:
\[\rho_{\tau}(x) = \rho_{\tau}(y) \quad  \Rightarrow \quad  \forall s \geq \tau, \ \rho_{s}(x)  =  \rho_{s}(y) . \]
\end{lm}

\begin{proof}
On the event $\{ \rho_{\tau}(x) = \rho_{\tau}(y) \}$, the processes $(X^{\delta}_{t}(x))_{t \geq 0}  := (\rho_{t}(x)^{2})_{t \geq 0}$ and $(X^{\delta}_{t}(y))_{t \geq 0} := (\rho_{t}(y)^{2})_{t \geq 0}$ both satisfy, on $[\tau, + \infty)$, the SDE:
\[ X_{t} =  \rho_{\tau}(x)^{2} + 2 \int_{\tau}^{t} \sqrt{X_{s}} dB_{s} + \delta (t-\tau) . \]
By pathwise uniqueness of this SDE (see \cite{revuz2013continuous}, Theorem (3.5), Chapter IX), we deduce that, a.s. on the event $\{ \rho_{\tau}(x) = \rho_{\tau}(y) \}$,  $X_{t}(x) = X_{t}(y)$, hence $\rho_{t}(x)  =  \rho_{t}(y)$ for all $t \geq \tau$.
\end{proof} 

Now we are able to prove Prop. \ref{derflow}.

\begin{proof}[Proof of Proposition \ref{derflow}]
Let $t > 0$ and $x > 0$ be fixed. First remark that:
\[ \mathbb{P}(T_{0}(x) = t) = 0. \]
Indeed, if $\delta > 0$, then :
\[ \mathbb{P} ( T_{0}(x) = t ) \leq  \mathbb{P} ( \rho_{t}(x) =0) \]
and the RHS is zero since the law of $\rho_{t}(x)$ has no atom on $\mathbb{R}_{+}$ (it has density $p^{\delta}_{t}(x,\cdot)$ w.r.t. Lebesgue measure on $\mathbb{R}_{+}$, where $p^{\delta}_{t}$ was defined in equation (\ref{sgpdensity}) above). On the other hand, if $\delta =0$, then $0$ is an absorbing state for the process $\rho$, so that, for all $s \geq 0$:
\[  \mathbb{P}(T_{0}(x) \leq s) = \mathbb{P} ( \rho_{s}(x) =0) \]
and the RHS is continuous in $s$ on $\mathbb{R}_{+}$, since it is given by $\exp(-\frac{x^{2}}{2s})$ (see \cite{revuz2013continuous}, Chapter XI, Corrolary 1.4). Hence, also in the case $\delta = 0$ the law of $T_{0}(x)$ has no atom on $\mathbb{R}_{+}$. Hence, a.s., either $t<T_{0}(x)$ or $t>T_{0}(x)$.

First suppose that  $t < T_{0}(x)$. A.s., the function $y \mapsto T_{0}(y)$ is continuous at $x$, so there exists a rational number $y \in [0,x)$ such that $t < T_{0}(y)$; since, by Remark (\ref{ascont}), $t \mapsto \rho_{t}(y)$ is continuous, there exists $\epsilon \in \mathbb{Q}_{+}^{*}$ such that $t < T_{\epsilon}(y)$. By monotonicity of $z \mapsto \rho(z)$, for all $s \in [0,t]$ and $z \geq y$,  we have:
\[\rho_{s}(z) \geq \rho_{s}(y) \geq \epsilon . \]
Hence, recalling Corollary \ref{assde}, for all $s \in [0,t]$ and $h \in \mathbb{R}$ such that $|h| < |x-y|$:
\[ \rho_{s}(x+h) = x + h + \int_{0}^{s} \frac{\delta-1}{2}\frac{du}{\rho_{u}(x+h)} + B_{s}. \]
Hence, setting $\eta^{h}_{s}(x):=\frac{\rho_{s}(x+h) - \rho_{s}(x)}{h}$,we have:
\[ \forall s \in [0,t], \qquad \eta^{h}_{s}(x) = 1 - \frac{\delta - 1}{2}  \int_{0}^{t} \frac{\eta^{h}_{u}(x)}{\rho_{u}(x) \rho_{u}(x+h)} du \]
so that :
\[ \eta^{h}_{t}(x) = \exp \left( \frac{1-\delta}{2} \int_{0}^{t} \frac{ds}{\rho_{s}(x) \rho_{s}(x+h)} \right). \]
Note that, for all $s \in [0,t]$ and $h \in \mathbb{R}$ such that $|h| < |x-y|$, we have $(s, x + h) \in [0, T_{\epsilon}(y)) \times (y, + \infty) \subset \mathcal{U}$. Hence, by Lemma \ref{bicont}, we have, for all $s \in [0,t]$ 
\[ \rho_{s}(x+h) \underset{h \to 0}{\longrightarrow} \rho_{s}(x) \]
with the domination property:
\[ \frac{1}{\rho_{s}(x) \rho_{s}(x+h)} \leq \epsilon^{-2} \]
valid for all $|h| < |x-y|$. Hence, by the dominated convergence theorem, we deduce that:
\[   \eta^{h}_{t}(x) \underset{h \to 0}{\longrightarrow} \exp \left( \frac{1-\delta}{2} \int_{0}^{t} \frac{ds}{\rho_{s}(x)^{2}} \right) \]
which yields the claimed differentiability of $\rho_{t}$ at $x$.

We now suppose that $t >T_{0}(x)$. Since the function $y \mapsto T_{0}(y)$ is a.s. continuous at $x$, a.s. there exists $y > x, \ y \in \mathbb{Q}$, such that $ t > T_{0}(y) $. By Remark (\ref{ascont}), the function $t \mapsto \rho_{t}(y)$ is continuous, so that $\rho_{T_{0}(y)}(y) = 0$.
By monotonicity of $z \mapsto \rho(z)$, we deduce that, for all $z \in [0,y]$, we have :
\[ \rho_{T_{0}(y)}(z) = 0 . \]
By Lemma \ref{traj}, we deduce that, leaving aside some event of proability zero,  all the trajectories $(\rho_{t}(z))_{t \geq 0}$ for $z \in [0,y]\cap \mathbb{Q}$ coincide from time $T_{0}(y)$ onwards. In particular, we have:
\[ \forall z \in [0,y]\cap \mathbb{Q}, \qquad \rho_{t}(z) = \rho_{t}(x) . \]
Since, moreover, the function $z \mapsto \rho_{t}(z)$ is nondecreasing, we deduce that it is constant on the whole interval $[0,y]$:
\[  \forall z \in [0,y], \qquad \rho_{t}(z) = \rho_{t}(x) . \] 
In particular, the function $z \mapsto \rho_{t}(z)$ has derivative $0$ at $x$. This concludes the proof.
\end{proof}

{\bf Acknowledgements.} I would like to thank Lorenzo Zambotti, my Ph.D. advisor, for all the time he patiently devotes in helping me with my research. I would also like to thank Thomas Duquesne and Nicolas Fournier, who helped me solve a technical problem, as well as Yves Le Jan for a helpful discussion on the Bessel flows of low dimension, and Lioudmila Vostrikova for answering a question on this topic.

\bibliography{Article-BEL-Bessel_2}

\begin{thebibliography}{13}
\providecommand{\natexlab}[1]{#1}
\providecommand{\url}[1]{\texttt{#1}}
\expandafter\ifx\csname urlstyle\endcsname\relax
  \providecommand{\doi}[1]{doi: #1}\else
  \providecommand{\doi}{doi: \begingroup \urlstyle{rm}\Url}\fi

\bibitem[Arnaudon et~al.(2017)Arnaudon, Li, et~al.]{arnaudon2017reflected}
M.~Arnaudon, X.-M. Li, et~al.
\newblock Reflected brownian motion: selection, approximation and
  linearization.
\newblock \emph{Electronic Journal of Probability}, 22, 2017.

\bibitem[Bismut(1984)]{bismut1984large}
J.-M. Bismut.
\newblock \emph{Large deviations and the Malliavin calculus}, volume~45.
\newblock Birkhauser, 1984.

\bibitem[Burdzy et~al.(2009)]{burdzy2009differentiability}
K.~Burdzy et~al.
\newblock Differentiability of stochastic flow of reflected brownian motions.
\newblock \emph{Electron. J. Probab}, 14\penalty0 (75):\penalty0 2182--2240,
  2009.

\bibitem[Cerrai(2001)]{cerrai2001second}
S.~Cerrai.
\newblock \emph{Second order PDE's in finite and infinite dimension: a
  probabilistic approach}, volume 1762.
\newblock Springer Science \& Business Media, 2001.

\bibitem[Da~Prato and Zabczyk(1996)]{da1996ergodicity}
G.~Da~Prato and J.~Zabczyk.
\newblock \emph{Ergodicity for infinite dimensional systems}, volume 229.
\newblock Cambridge University Press, 1996.

\bibitem[Deuschel and Zambotti(2005)]{deuschel2005bismut}
J.-D. Deuschel and L.~Zambotti.
\newblock Bismut--elworthy's formula and random walk representation for sdes
  with reflection.
\newblock \emph{Stochastic processes and their applications}, 115\penalty0
  (6):\penalty0 907--925, 2005.

\bibitem[Elworthy and Li(1994)]{ELWORTHY1994252}
K.~Elworthy and X.~Li.
\newblock Formulae for the derivatives of heat semigroups.
\newblock \emph{Journal of Functional Analysis}, 125\penalty0 (1):\penalty0 252
  -- 286, 1994.
\newblock ISSN 0022-1236.
\newblock \doi{http://dx.doi.org/10.1006/jfan.1994.1124}.
\newblock URL
  \url{http://www.sciencedirect.com/science/article/pii/S0022123684711244}.

\bibitem[Hairer and Mattingly(2016)]{hairer2016strong}
M.~Hairer and J.~Mattingly.
\newblock The strong feller property for singular stochastic pdes.
\newblock \emph{arXiv preprint arXiv:1610.03415}, 2016.

\bibitem[Revuz and Yor(2013)]{revuz2013continuous}
D.~Revuz and M.~Yor.
\newblock \emph{Continuous martingales and Brownian motion}, volume 293.
\newblock Springer Science \& Business Media, 2013.

\bibitem[Tsatsoulis and Weber(2016)]{tsatsoulis2016spectral}
P.~Tsatsoulis and H.~Weber.
\newblock Spectral gap for the stochastic quantization equation on the
  2-dimensional torus.
\newblock \emph{arXiv preprint arXiv:1609.08447}, 2016.

\bibitem[Varadhan and Williams(1985)]{varadhan1985brownian}
S.~R. Varadhan and R.~J. Williams.
\newblock Brownian motion in a wedge with oblique reflection.
\newblock \emph{Communications on pure and applied mathematics}, 38\penalty0
  (4):\penalty0 405--443, 1985.

\bibitem[Vostrikova(2009)]{vostrikova2009regularity}
L.~Vostrikova.
\newblock On regularity properties of bessel flow.
\newblock \emph{Stochastics: An International Journal of Probability and
  Stochastics Processes}, 81\penalty0 (5):\penalty0 431--453, 2009.

\bibitem[Zambotti(2017)]{zambotti2017random}
L.~Zambotti.
\newblock \emph{Random Obstacle Problems: {\'E}cole d'{\'E}t{\'e} de
  Probabilit{\'e}s de Saint-Flour XLV-2015}, volume 2181.
\newblock Springer, 2017.

\end{thebibliography}

\end{document}